\documentclass[a4paper,12pt]{amsart}
\usepackage{amssymb}
\usepackage{ifthen}
\usepackage{xargs}
\usepackage{a4wide}
\usepackage{url}
\usepackage{xcolor}
\usepackage{hyperref}
\usepackage[all]{xy}
\usepackage{todonotes}
\newcommand{\R}{\mathbb{R}} 
\newcommand{\Q}{\mathbb{Q}} 
\newcommand{\uhp}{\mathbb{H}} 
\newcommand{\C}{\mathbb{C}} 
\newcommand{\Z}{\mathbb{Z}} 
\newcommand{\F}{\mathbb{F}} 
\newcommand{\N}{\textup{N}} 

\newcommand{\Zhat}{\hat{\mathbb{Z}}}
\newcommand{\adeles}{\mathbb{A}}
\newcommand*{\domain}{\mathbb{D}}

\newcommand*{\reg}{\mathrm{reg}}
\newcommand*{\laplace}{\Delta}

\newcommand{\bs}{\backslash}

\newcommand{\calB}{\mathcal{B}}
\newcommand{\calC}{\mathcal{C}}

\newcommand{\calF}{\mathcal{F}}

\newcommand{\calH}{\mathcal{H}}

\newcommand{\calL}{\mathcal{L}}

\newcommand{\calO}{\mathcal{O}}

\newcommand{\calQ}{\mathcal{Q}}

\newcommand{\frakA}{\mathfrak{A}}
\newcommand{\fraka}{\mathfrak{a}}
\newcommand{\frakb}{\mathfrak{b}}

\newcommand{\frakp}{\mathfrak{p}}

\newcommand{\e}[1]{e\left(#1\right)} 

\newcommand{\smallabcd}{\left(\begin{smallmatrix}a & b \\ c & d\end{smallmatrix}\right)}

\newcommand{\Tmatrix}{\begin{pmatrix}1 & 1 \\ 0 & 1\end{pmatrix}}
\newcommand{\Smatrix}{\begin{pmatrix}0 & -1 \\ 1 & 0\end{pmatrix}}

\newcommand{\twomat}[4]{\begin{pmatrix}
                        #1 & #2 \\ #3 & #4 
                       \end{pmatrix}}

\newcommand{\inverse}{^{-1}}

\newcommand{\abs}[1]{\left\vert#1\right\vert}

\DeclareMathOperator{\Aut}{Aut}
\DeclareMathOperator{\SL}{SL}
\DeclareMathOperator{\GL}{GL}

\DeclareMathOperator{\Mp}{Mp}
\DeclareMathOperator{\Cl}{Cl}
\DeclareMathOperator{\tr}{tr}

\DeclareMathOperator{\Gal}{Gal}
\DeclareMathOperator{\divisor}{div}
\DeclareMathOperator{\Og}{O}
\DeclareMathOperator{\SO}{SO}
\DeclareMathOperator{\Gr}{Gr}

\DeclareMathOperator{\GSpin}{GSpin}

\DeclareMathOperator{\ord}{ord}

\DeclareMathOperator{\vol}{vol}

\DeclareMathOperator{\supp}{supp}
\DeclareMathOperator{\sgn}{sgn}
\DeclareMathOperator{\charfn}{char}

\DeclareMathOperator{\CT}{CT}

\DeclareMathOperator{\res}{res}
\DeclareMathOperator{\del}{\partial}
\DeclareMathOperator{\delbar}{\overline{\partial}}

\newtheorem{theorem}{Theorem}[section]
\newtheorem{proposition}[theorem]{Proposition}
\newtheorem{lemma}[theorem]{Lemma}
\newtheorem{corollary}[theorem]{Corollary}

\theoremstyle{definition}
\newtheorem{definition}{Definition}[section]

\newtheorem{remark}{Remark}[section]

\newcommand{\G}{\mathbb{G}}
\newcommand{\bp}{{b^+}}
\newcommand{\bm}{{b^-}}

\newcommand{\zperp}{{z^\perp}}

\newcommand{\pre}[1]{{#1^{\prime}}}
\newcommand{\prep}[1]{{#1^{\prime\, +}}}
\renewcommand{\S}{\mathcal{S}}

\newcommand{\kronecker}[2]{\left(\frac{#1}{#2}\right)}
\newcommand{\smallabcmat}{\left(\begin{smallmatrix} b/2 & -a \\ c & -b/2 \end{smallmatrix}\right)}
\newcommand{\abcmat}{\begin{pmatrix} \frac{b}{2} & -a \\ c & -\frac{b}{2} \end{pmatrix}}
\newcommand{\zmatrix}{\begin{pmatrix}  -z & -1 \\ z^2 & z \end{pmatrix}}

\DeclareMathOperator{\Diff}{Diff}

\DeclareMathOperator{\sym}{sym}

\title[CM values of Borcherds products]{On CM values of Borcherds products\\ and harmonic weak Maa\ss{} forms of weight one}
\author{Stephan Ehlen}
\date{\today}
\thanks{This research has been supported by DFG grant BR-2163/2-1.}

\begin{document}
\begin{abstract}
  We show that the values of Borcherds products on Shimura varieties of orthogonal type
  at certain CM points are given in terms of coefficients of the holomorphic part
  of weight one harmonic weak Maa\ss{} forms.
  Furthermore, we investigate the arithmetic properties of these coefficients.
  As an example, we obtain an analog of the Gross-Zagier theorem on singular moduli.
\end{abstract}
\maketitle
\tableofcontents

\section{Introduction}
The celebrated theorem of Gross and Zagier on singular moduli \cite{grosszagier-singularmoduli}
gives an explicit prime factorization of the norm of the modular function
$j(z)$ at imaginary quadratic arguments in the complex upper half-plane $\uhp$.
These values are clasically called singular moduli \cite{zagiertraces}.
Let $d < 0$ be a negative fundamental discriminant and denote by
$\calQ_{d}$ the set of positive definite integral binary quadratic forms
of discriminant $d$. For every $Q=[a,b,c]=ax^{2}+bxy+cy^{2} \in \calQ_{d}$
we denote by $\alpha_{Q} \in \uhp$ the unique root of $Q(\tau,1)=0$ in $\uhp$.
The group $\SL_2(\Z)$ acts on $\calQ_{d}$ in the standard way
and two forms $P,Q \in \calQ_{d}$ are equivalent under this action if and only if
the points $\alpha_{P}$ and $\alpha_{Q}$ are equivalent under the action of
$\SL_2(\Z)$ via linear fractional transformations on $\uhp$.
The value $j(\alpha_{Q})$ is an algebraic integer of degree $h_{d}$ over $\Q$,
generating the Hilbert class field $H$ of $k_{d} = \Q(\sqrt{d})$.
Here, $h_{d}$ denotes the class number of $k_{d}$. The values
$\{ j(\alpha_{Q})\ ;\ Q \in \SL_2(\Z) \bs \calQ_{d} \}$ form a full system
of Galois conjugates.

Consider the modular function
\begin{equation*}
  \Psi(z,d) = \prod_{Q \in \SL_2(\Z) \bs \calQ_{d}} (j(z)-j(\alpha_{Q}))^{4/w_{d}},
\end{equation*}
for $z \in \uhp$, and where $w_{d}$ is the number of roots of unity in $k_{d}$.
Gross and Zagier proved \cite{grosszagier-singularmoduli} that
for a negative fundamental discriminant $D$ coprime to $d$, we have
$$
\prod_{Q \in \SL_2(\Z) \bs \calQ_{D}}\Psi(\alpha_{Q},d)^{2/{w_{D}}}
= \pm \prod_{\substack{x\in \Z,\, n,n'>0\\ 4nn'=dD -x^2}}n^{\epsilon(n')},
$$
where $\epsilon(n') = \pm 1$.

The theorem of Gross and Zagier can be understood
in the context of Borcherds products. These are certain meromorphic modular forms
on orthogonal groups with an infinite product expansion and a divisor which is a linear combination
of so-called ``special divisors''. Borcherds products can be obtained
via a singular theta correspondence \cite{boautgra}.

The theorem of Gross and Zagier has been generalized by Schofer \cite{schofer} to all Borcherds products
on orthogonal groups.
He considered the values of Borcherds products at generalizations of the points $\alpha_{Q} \in \uhp$
to orthogonal modular varieties. We call these points CM points (CM for complex multiplication) and the
value of a function at a CM point is called a CM value (see Section \ref{sec:cmcycles} for definitions).

An interesting interpretation of Schofer's result is
that it expresses the \emph{norm of a CM value} of a Borcherds product
in terms of the coefficients of the holomorphic part of a certain harmonic weak Maa\ss{} form
$\pre{E_{\frakA}}(\tau)$ of weight one attached to a genus $\frakA$ of binary quadratic forms.
Employing the theory of harmonic weak Maa\ss{} forms, as developed by
Bruinier and Funke \cite{brfugeom}, we are able to show that also the \emph{CM value itself}
is a linear combination of the coefficients of a certain harmonic weak Maa\ss{} form of weight one.

\subsection{A theorem on singular moduli}
\label{sec:theor-sing-moduli}
Let $M_{k}^{!}(\Gamma_0(\abs{D}),\chi_{D})$
be the space of weakly holomorphic modular forms of weight $k$
for $\Gamma_0(\abs{D}) \subset \SL_2(\Z)$ with Nebentypus $\chi_D=\kronecker{D}{\cdot}$, the Kronecker symbol.
These are meromorphic modular forms whose poles are supported at the cusps.
The space $M_{k}^{!}(\Gamma_0(\abs{D}),\chi_{D})$ is contained in the space $\calH_{k}(\Gamma_0(\abs{D}),\chi_{D})$
of harmonic weak Maa\ss{} forms of weight $k$ and Nebentypus $\chi_D$
(we refer to Section \ref{sec:harmonic-weak-maass} for the precise definitions).
An element $f \in \calH_{k}(\Gamma_0(\abs{D}),\chi_{D})$
admits a unique decomposition $f=f^{+}+f^{-}$ into a \emph{holomorphic part} $f^{+}$
and a \emph{non-holomorphic part} $f^{-}$.
We write the Fourier expansion at the cusp $\infty$ of the holomorphic part
of $f \in \calH_{k}(\Gamma_0(\abs{D}),\chi_{D})$ as
\[
 f^{+}(\tau) = \sum_{n \gg -\infty} c_f^{+}(n) q^{n}.
\]
If $f$ is weakly holomorphic, we will frequently just write $c_f(n)$ for the coefficients.
We call the Fourier polynomial
\[
  P_{f}= \sum_{n \leq 0} c_f^{+}(n) q^{n}
\]
the \emph{principal part} of $f$.

There is an antilinear differential operator
$\xi=\xi_{k}:\mathcal{H}_{k}(\Gamma_0(\abs{D},\chi_{D}) \to M^{!}_{2-k}(\Gamma_0(\abs{D},\chi_{D})$ defined by
\begin{equation}
  \xi(f)(\tau)
  := -2 v^{k} \overline{\frac{\partial}{\partial\overline{\tau}} f(\tau)}
\end{equation}
for $f \in \calH_{k}(\Gamma_0(\abs{D},\chi_{D})$.
It was shown by Bruinier and Funke that this map is surjective \cite[Theorem 3.7]{brfugeom}.

The harmonic weak Maa\ss{} form $\pre{E_{\frakA}}$ above has the property
$\xi(\pre{E_{\frakA}})=E_{\frakA}$, where $E_{\frakA}$ is the normalized
Eisenstein series attached to the genus $\frakA$.
We refer to \cite{KudlaYangEisenstein, kryderivfaltings, kudla-annals-central-der}
for details.
The following analog of the Gross-Zagier theorem is an example illustrating our result
(Theorem \ref{thm:value-phiz}) on CM values of Borcherds products.

For a quadratic form $Q \in \calQ_{D}$,
we denote by $\theta_{Q}$ the corresponding theta function,
that is,
\[
  \theta_{Q}(\tau) = \sum_{x,y \in \Z} e\left( Q(x,y) \tau \right).
\]
It is a modular form of weight $1$ for the group $\Gamma_{0}(\abs{D})$ with Nebentypus
$\chi_{D}$ and depends only on the class of $Q$ with respect to the action of $\SL_{2}(\Z)$.
By the surjectivity of $\xi$, for every $Q \in \calQ_{D}$ there is a
$\pre{\theta_{Q}} \in \calH_1(\Gamma_0(\abs{D}),\chi_D)$ such that
$\xi (\pre{\theta_{Q}}) = \theta_{Q}$. We will simply write $c_{Q}^{+}(n)$ for the Fourier
coefficients of the holomorphic part of $\pre{\theta_Q}$.
We require the following technical condition on $D$:
assume that $D \equiv 1 \bmod 4$ and that
\begin{equation}
  \label{eq:DassIntro}
    \prod_{\substack{p \text{ prime} \\ p \mid D}}\left( 1 + \frac{1}{p} \right) \leq 3,
\end{equation}
which is, for instance,
always satisfied for discriminants which are the product of at most two distinct primes.
This assumption, which is directly obtained using the Sturm bound for $M_{1}(\Gamma_0(\abs{D}),\chi_D)$,
gives an explicit bound for the minimal length of the principal part of the forms $\pre{\theta_{Q}}$
that we are allowed to ``choose''.

\begin{theorem} \label{thm:GZ-intro}
 Let $D \equiv 1 \bmod 4$ be an odd negative fundamental discriminant
 satisfying condition \eqref{eq:DassIntro}.
 \begin{enumerate}
   \item For every $Q \in \calQ_{D}$,
     we can choose $\pre{\theta_{Q}} \in \calH_1(\Gamma_0(\abs{D}),\chi_D)$ with $\xi(\pre{\theta_{Q}})=\theta_{Q}$,
     such that $c^{+}_{Q}(n) = 0$ for all $n \leq \frac{D}{4}$.
   \item For every $\pre{\theta_{Q^{2}}}$ satisfying $(i)$,
     and every negative fundamental discriminant $d$ coprime to $D$,
         we have
         \begin{equation*}
                        \log \abs{ \Psi(\alpha_{Q},d) }
                         =  -\frac{1}{4} \sum_{\substack{m \in \Z \\ m \equiv d \bmod 2}}
                        \delta(m)\, c^{+}_{Q^{2}}\left(\frac{Dd - m^{2}}{4} \right).
         \end{equation*}
         Here, $Q^2$ denotes the square (class) of $Q$ with respect to the composition of
         binary quadratic forms and
\begin{equation*}
  \delta(\mu) =
    \begin{cases}
      2, & \text{if } \mu \equiv 0 \bmod \abs{D},\\
      1, & \text{if } \mu \not\equiv 0 \bmod \abs{D}.
    \end{cases}
\end{equation*}
 \end{enumerate}
\end{theorem}
The preceding theorem follows from our general result on Borcherds products
(Theorem \ref{thm:value-phiz}) together with the results in Section
\ref{sec:princ-parts-preim}.
We will work this out in detail in Section \ref{sec:gross-zagi-situ}.

\subsection{The coefficients of the holomorphic part}
\label{sec:coeff-holom-part-1}
In contrast to the function $\pre{E_{\frakA}}$,
there are no known explicit formulas for the coefficients $c_{Q}^{+}(n)$.
The function $\pre{E_{\frakA}}$ is special since it can be \emph{explicitly} obtained
as the derivative of an incoherent Eisenstein series of weight one.
This means that Theorem \ref{thm:GZ-intro} encodes new and interesting
information about both sides of the equation:
starting with a weakly holomorphic modular form $f$ with integral principal part,
we know that the CM values of the modular function $\Psi(z,f)$ are algebraic.
This clearly suggests that the coefficients of the holomorphic part of $\pre{\theta_{Q}}(\tau)$
are of arithmetic nature and should be linear combinations of terms the form $\log|\alpha|$,
where $\alpha$ is an algebraic number.

Employing a certain seesaw identity, we are able to relate the individual coefficients to
CM values of modular functions on Shimura curves, revealing their arithmetic nature.
We will normalize the functions $\pre{\theta_{Q}}$ in the following way.
Write $\theta_{Q}(\tau) = E_{\frakA}(\tau) + g_{Q}(\tau)$,
where $g_{Q} \in S_1(\Gamma_{0}(\abs{D}),\chi_{D})$
is a cusp form attached to $Q$, with rational Fourier coefficients,
and $\frakA$ denotes the genus of $Q$.
We will write
$\pre{\theta_{Q}}(\tau) = \pre{E_{\frakA}}(\tau) + \pre{g_{Q}}(\tau)$ with
$\xi(\pre{g_{Q}}(\tau))=g_{Q}(\tau)$. We will assume that $D$ is square-free
and write its prime factorization as $D=-p_{1}\cdots p_{r}$.
We denote the corresponding genus characters by $\chi_{p_{1}},\ldots,\chi_{p_{r}}$
and also regard them as Dirichlet characters given by the corresponding Kronecker
characters.

Our second main result is summarized in the following theorem.
\begin{theorem}
  \label{thm:coeffs-intro}
  Let $D<0$ with $D \equiv 1 \bmod 4$ be a fundamental discriminant.
  For every $Q \in \calQ_{D}$, there is a harmonic weak Maa\ss{} form
  $\pre{\theta_{Q}} \in \calH_{1}(\Gamma_0(\abs{D}),\chi_D)$ with
  the following properties.
  \begin{enumerate}
  \item If $D$ satisfies \eqref{eq:DassIntro}, we can require that condition $(i)$ of Theorem \ref{thm:GZ-intro} holds.
  \item We have $\xi(\pre{\theta_{Q}})=\theta_{Q}$.
  \item The preimages $\pre{\theta_{Q}}$ can be chosen to be compatible
        with the Siegel-Weil formula, that is
        \[
          \sum_{R \in \SL_2(\Z)\bs\calQ_{D}} \pre{\theta_{{R}^2Q}} = h_{D}\, \pre{E_{\frakA}},
        \]
        where $\frakA$ denotes the genus of $Q$.
  \item If $\chi_{p_{i}}(n)=\chi_{p_{i}}(\frakA)=1$ for any $i$ with $p_{i} \equiv 3 \bmod 4$,
        then $c_{Q}^{+}(n) = 0$.
  \item For $n < 0$ with $\chi_{p_{i}}(n) = - \chi_{p_{i}}(\frakA)$ for all $i$
        with $p_{i} \equiv 3 \bmod 4$, we have
        \[
          c_{Q}^{+}(n) = \sum_{R \in \frakA} a_{R}(n) (g_{Q},g_{R}),
        \]
    where $(\cdot,\cdot)$ denotes the Petersson inner product on $\Gamma_0(\abs{D})$
    and the coefficients $a_{R}(n) \in \Q$ are explicit constants.
  \item The constant term of $\prep{\theta_{Q}}$ is equal to
    \[
      c_{Q}^{+}(0) = - \sum_{n>0}c_{Q}^{+}(-n)\rho(n) - 2 \frac{\Lambda'(\chi_{D},0)}{\Lambda(\chi_{D},0)}.
    \]
    Here, $\rho(n) = \#\{\frakb \subset \calO_{D} \mid\ \N(\frakb)=n \}$, where $\calO_{D} \subset k_{D}$
    is the ring of integers in $k_{D}$, and
    $\Lambda(\chi_{D},s)$ denotes the completed $L$-function for $\chi_{D}$
    (see Theorem \ref{thm:EpreFourier}).
  \item Finally, for $n > 0$ with $\chi_{p_{i}}(n) = - \chi_{p_{i}}(\frakA)$ for all $i$
        with $p_{i} \equiv 3 \bmod 4$, we have
    \[
      c_{Q}^{+}(n) = \log\abs{\alpha(Q,n)} - R(n,\pre{\theta_{Q}}),
    \]
    for some algebraic numbers $\alpha(Q,n) \in \overline{\Q}$, and $R(n,\pre{\theta_{Q}})$
    is given by a rational linear combination of the coefficients in the principal part of $\pre{\theta_{Q}}$.
  \end{enumerate}
\end{theorem}
The statement of the preceding theorem is essentially contained in Theorem \ref{thm:pos-coeffs},
which gives the more precise formulation in terms of
vector-valued modular forms.

\begin{remark}
  The Petersson inner products appearing in Theorem \ref{thm:coeffs-intro} (v)
  can be expressed in terms of CM values of modular forms.
  To show this, we will use a seesaw identity
  involving the theta correspondences between $\SL_{2}$ and the orthogonal groups $\SO(2,2)$
  and $\SO(2)$, respectively.
  This results in an expression involving CM values
  of a modular form of parallel weight $1/2$ for a congruence subgroup of $\SL_{2}(\Z) \times \SL_2(\Z)$,
  which arises as the lift of a constant function.
  In Section \ref{sec:princ-part:-peterss}, we will see how this implies that $(g_{Q},g_{R})$
  can be written as a linear combination of logarithms of absolute values of algebraic numbers.
  For prime discriminants, we will show
  that the modular form of parallel weight $1/2$ is in fact $\eta(\tau_{1})\eta(\tau_{2})$,
  where $\eta$ is the Dedekind $\eta$ function.
  In the case of prime discriminants, Duke and Li \cite{DukeLiMock} obtained the same result
  using the connection of weight one modular forms to Galois representations.

  Using these results, Theorem \ref{thm:coeffs-intro}
  provides an explicit formula for the principal part of $\pre{\theta_{Q}}$.
  This also gives a formula for $R(n,\pre{\theta_{Q}})$.
  Moreover, in Section \ref{sec:coeff-holom-part}, we will show
  that the algebraic numbers $\alpha(Q,n)$
  are given by CM values of modular functions on Shimura curves.
\end{remark}

\subsection{The main ideas}
\label{sec:main-ideas}
We will now indicate how to prove these two results.
As in the main part of the paper, we will describe this in terms of vector-valued modular forms
(see Section \ref{sec:harmonic-weak-maass} for definitions).
We consider an even lattice $L$ of type $(2,n)$ with quadratic form $q:L \rightarrow \Z$ and the
corresponding rational quadratic space $V=L \otimes \Q$.
We denote the dual lattice of $L$ by $L'$.
The discriminant group $L'/L$ is a finite abelian group.
The symmetric domain $\domain$ associated with $V \otimes \R$ can be realized as the Grassmannian manifold
of $2$-dimensional positive definite subspaces of $V \otimes \R$.
We denote by $S_{k,L} \subset M_{k,L} \subset M_{k,L}^{!} \subset \calH_{k,L}$
the corresponding spaces of vector-valued cusp forms,
holomorphic, weakly holomorphic and harmonic weak Maa\ss{} forms of weight $k$ transforming with
the Weil representation $\rho_{L}$ associated with $L$. These functions take values in a finite dimensional
space $A_{L}$ of Schwartz functions spanned by the characteristic functions $\phi_{\mu}$ of
the cosets $\mu \in L'/L$. This space is isomorphic to the group ring $\C[L'/L]$.

Moreover, we will denote the lattice given by $L$ together with the quadratic form $-q$ by $L^{-}$.
The corresponding representation $\rho_{L^{-}} = \overline{\rho}_{L}$ is the dual of $\rho_{L}$.
We will write $\langle \cdot, \cdot \rangle$ for the bilinear scalar product on $A_{L}$,
such that $\langle c_{1} \phi_{\mu}, c_{2} \phi_{\nu} \rangle = c_{1} c_{2} \delta_{\mu,\nu}$.
We let $\Theta_{L}(\tau,z)$ be the vector-valued
Siegel theta function for $\tau \in \uhp$ and $z \in \domain$. It is a non-holomorphic
modular form in both variables and is invariant under a certain arithmetic subgroup $\Gamma$ of $\Og(V)$ in $z$.
In $\tau$, it has weight $(2-n)/2$ and transforms with
the Weil representation of the metaplectic group $\Mp_{2}(\Z)$.

For $f \in M^{!}_{1-n/2,L}$, we can consider the theta integral
\[
 \Phi(z,f) = \int_{\SL_2(\Z) \bs \uhp}^{\reg}
             \langle f(\tau), \overline{\Theta_L(\tau,z)}  \rangle v^{\frac{2-n}{2}} d\mu(\tau),
\]
where $d\mu(\tau)=\frac{dudv}{v^{2}}$ denotes the invariant measure on $\uhp$ and $\tau=u+iv$.
The integral above is divergent but can be regularized as shown by Borcherds \cite{boautgra}
using the method of Harvey and Moore \cite{hm-bps}.
The regularized integral, indicated by the superscript ``$\reg$'',
defines a $\Gamma$-invariant real analytic function on $\domain \setminus Z(f)$ with
logarithmic singularities along a special divisor $Z(f)$.
This construction has been used by Borcherds to show the existence of a meromorphic modular form $\Psi(z,f)$
on $\domain$ with divisor $Z(f)$.

A CM point $z_{U} \in \domain$ is given by a $2$-dimensional
\emph{rational} positive definite subspace $U \subset V$.
Let $P=L \cap U$ and $N = L \cap U^{\perp}$.
Then $P$ is a lattice of rank 2 and is positive definite, whereas $N$
is negative definite and of rank $n$.
Assume for simplicity that the lattice $L$ splits as
$L=P \oplus N$.

The idea we will use to prove Theorem \ref{thm:GZ-intro} is the following:
Using Stoke's theorem, we obtain a formula for the CM value $\Phi(z_{U},f)$ in terms of
only finitely many coefficients of a harmonic weak Maa\ss{} form
$\pre{\Theta_P}$ with $\xi(\pre{\Theta_P})=\Theta_P$,
the coefficients of $f$ and representation numbers for the lattice $N$.
In fact, similar to work of Bruinier and Yang \cite{bryfaltings},
we can extend the theta lift to harmonic weak Maa\ss{} forms and obtain the following result.
\begin{theorem}\label{thm:value-phiz-intro}
  Let $f \in H_{1-n/2,L}$.
  Then we have
  \begin{align*}
    \Phi(z_{U},f) &= \CT \left( \langle f^+(\tau), \Theta_{N^{-}}(\tau) \otimes \prep{\Theta_P}(\tau) \rangle \right) \\
    &\quad + \int_{\SL_2(\Z) \bs \uhp}^{\reg} \langle \xi(f),\Theta_{N^{-}}(\tau) \otimes \pre{\Theta_P}(\tau) \rangle d\mu(\tau).
  \end{align*}
\end{theorem}
Here, for any $q$-series $F = \sum_{n \in \Z} a(n) q^{n}$ we let $\CT( F )=a_{0}$.
Moreover, the assumption $f \in H_{k,L} \subset \calH_{k,L}$ is equivalent to $\xi(f)\in S_{2-k,L^{-}}$.
Note that for weakly holomorphic $f$, the regularized integral on the right hand side vanishes
since $\xi(f)=0$.
Thus, Theorem \ref{thm:GZ-intro} is in fact an example of \ref{thm:value-phiz-intro}
for an appropriate choice of $L$ and $f$. In Section \ref{sec:gross-zagi-situ}
we will work through this example in detail.

Our approach to obtain a formula for the individual coefficients is the following.
For a $2$-dimensional positive definite lattice $P$
let $\Phi_{P}$ be the theta lift corresponding to $P$.
For a weakly holomorphic modular form $f \in M_{1,P}^{!}$ it is defined as
\[
  \Phi_{P}(f) = \int_{{\SL_2(\Z) \bs \uhp}}^{\reg} \langle f(\tau), \overline{\Theta_P(\tau)} \rangle\, v\, d\mu(\tau).
\]
This is the regularized Petersson scalar product of $f$ and $\Theta_P$
(note that the Grassmannian in this case consists
of only one point and thus, we could drop the argument $z \in \domain$).
Applying Theorem \ref{thm:value-phiz-intro}, we obtain that
\begin{equation}
  \label{eq:pairing-f-theta}
  \Phi_P(f) = \sum_{\beta \in P'/P}\sum_{n>0}c_{f}(-n,\beta)\, c^{+}_{P}(n,\beta)
            + \sum_{\beta \in P'/P}\sum_{n \geq 0} c_{f}(n,\beta)\, c^{+}_{P}(-n,\beta),
\end{equation}
where $c_{P}^{+}$ denotes the coefficients of $\prep{\Theta_P}$.
Denote the second term on the right-hand side by $R(f,\pre{\Theta_P})$.
This is essentially the second contribution to the right-hand side of
(vii) in Theorem \ref{thm:coeffs-intro}.

In order to obtain another expression for the first sum
in \eqref{eq:pairing-f-theta}, we apply a certain seesaw identity \cite{Kudla-seesaw}.
We will use this in connection with a variant \cite{marynaCMvals}
of an embedding trick employed by Borcherds \cite{boautgra}.
Namely, there is a lattice $M(P)$ of type $(2,1)$ with $P \subset M(P)$
and a weakly holomorphic modular form $h_f \in M_{\frac{1}{2},M(P)}^{!}$,
such that
\[
  \Phi_P(f) = \Phi_{M(P)}(h_{f}, z_{P})
  := \int_{{\SL_2(\Z) \bs \uhp}}^{\reg} \langle h_{f}(\tau), \overline{\theta_{M(P)}(\tau,z_{P})} \rangle\, v^{\frac{1}{2}}\, d\mu(\tau).
\]
Here, the point $z_{P}$ is the CM point corresponding to $P \otimes \Q \subset M(P) \otimes \Q$.

Now suppose that the principal part of $f$
in the components indexed by $\beta$ and $-\beta$ is equal to $q^{-m}$
and vanishes in all other components.
Then the first term in \eqref{eq:pairing-f-theta} is simply
$c^{+}_{P}(m,\beta)+c^{+}_{P}(m,-\beta)$, which is equal to
$2 c^{+}_{P}(m,\beta)$. This symmetry is enforced by the
action of the negative of the identity matrix in $\SL_{2}(\Z)$.
Thus, we obtain
\[
 2c^{+}_{P}(m,\beta) = \Phi_{M(P)}(h_{f}, z_{P}) - R(f,\pre{\Theta_P}).
\]
Using that $\Phi_{M(P)}(h_{f}, z_{P})$ is essentially the logarithm
of a meromorphic modular function on $\uhp$ \cite[Theorem 13.3]{boautgra}
it follows by CM theory \cite{shimauto} that the value $\Phi_{M(P)}(h_{f}, z_{P})$
is of the form $\log\abs{\alpha}$ for an algebraic $\alpha$,
which yields the desired result.

Unfortunately, we cannot always choose $f$ to be of this simple form
since the space of obstructions for $M_{1,P}^{!}$ is given by
$S_{1,P^{-}}$, which is in general non-empty.
This issue will be treated in detail in Section \ref{sec:coeff-holom-part}.

Note that the formulas we obtain are not yet in a form that can be used for direct computations.
However, with a refinement of the techniques we used, it should be possible to obtain
the contribution coming from the CM value of $\Phi_{M(P)}$ explicitly and to obtain
closed formulas for the coefficients $c_{P}^{+}(n,\beta)$.
We will treat this question in a sequel to this paper.

\section*{Acknowledgments}
I would like to thank J. Bruinier for his constant support, advise and many fruitful discussions.
Recently, I learned that there is some overlap with work of W. Duke and Y. Li \cite{DukeLiMock}.
In the case of prime discriminants and using different methods, they obtained Theorem \ref{thm:GZ-intro}
and similar statements for the coefficients of the holomorphic part of suitably normalized preimages.
I would like to thank Y. Li for sharing his ideas on their paper.
Moreover, I am indebted to C. Alfes and F. Str\"{o}mberg for
their many helpful comments on earlier versions of this paper.
Furthermore, I would like to thank J. Funke, S. Kudla and T. Yang
for interesting discussions on this and related subjects.

\section{Preliminaries}
\subsection{Shimura varieties}
Let $n \geq 0$ be an integer and let $V$ be a quadratic space over $\Q$ of
type $(2, n)$ with a non-degenerate quadratic form $Q$.
We write $(\cdot,\cdot)$ for the corresponding bilinear form.
The type of a quadratic space is $(p,q)$, if $p$ is the number of positive and $q$
the number of negative eigenvalues of the Gram matrix of any basis of $V$.

\begin{remark}
  Our setup is basically the same as in \cite{kudla-integrals, schofer, bryfaltings}.
  However, we warn the reader that we are working with a quadratic space of type $(2,n)$
  whereas Kudla's setup which also has been adopted by Bruinier, Yang and Schofer
  always uses type $(n,2)$ quadratic spaces.
  We have nevertheless preferred to use type $(2,n)$ as in \cite{boautgra} because
  the theta functions for lattices of type $(2,0)$ are holomorphic.
\end{remark}

We consider the group
\begin{equation}
  H: = \GSpin(V) = \{g \in C_V^0 \ \mid\
  g\ \text{invertible and}\ gVg\inverse =V\},
\end{equation}
as an algebraic group over $\Q$. Here, $C_V^0$ is the even Clifford algebra of $V$ \cite{kitqf}.
The group $H$ is a central extension of the
special orthogonal group $\SO_V$. One has the following exact sequence
of algebraic groups:
\begin{equation}
  1 \longrightarrow \G_m \longrightarrow H \longrightarrow \SO_V \longrightarrow 1.
\end{equation}
Here, $\G_{m}$ denotes the multiplicative (algebraic) group.

Consider the Grassmannian of oriented $2$-dimensional positive definite
subspaces of $V(\R)$, that is
\begin{equation}
  \domain: = \{z^{\pm} \ \big | \
  z \subset V(\R), \dim z=\bp, \ Q|_z > 0 \}.
\end{equation}
Here, for each subspace $z \subset V(\R)$,
we write $z^{\pm}$ for the two possible choices of orientation.
The Grassmannian has two connected components and each of them is
isomorphic to the symmetric space $\SO_V(\R)/C$, where
$C$ is a maximal compact subgroup of $\SO_V(\R)$.
The group $H(\R)$ acts naturally on $\domain$ and this action is transitive
by Witt's theorem.

We denote by $\adeles$ the adeles over $\Q$ and by $\adeles_{f}$ the
finite adeles. That is,
\[
  \adeles = \sideset{}{'}\prod_{p \leq \infty} \Q_{p},
  \quad \adeles_{f} = \sideset{}{'}\prod_{p < \infty} \Q_{p},
\]
where the $'$ indicates that we take the restricted product
with respect to the $p$-adic integers $\Z_{p}$, and $\Q_{p}$
denotes the field of $p$-adic numbers.

Let $K \subset H(\adeles_f)$ be an open compact subgroup. We write
$X_K$ for the associated Shimura variety
\begin{equation*}
  X_K = H(\Q) \backslash ( \domain \times H(\adeles_f)/K).
\end{equation*}

We have the following decomposition into connected Shimura varieties.
\begin{lemma}[{\cite[Lemma 5.13]{MilneShimuraVars}}]\label{lem:shimdecomp}
  Let $\mathcal{C}$ be a set of representatives for the double coset space
  $H(\Q)\backslash H(\adeles_f)/K$ and let $\domain^+$
  be a connected component of $\domain$. Then
  \begin{equation*}
    X_K \cong \bigsqcup_{h \in \mathcal{C}} \Gamma_h\backslash \domain^+,
  \end{equation*}
  where $\Gamma_g$ is the subgroup $hKh^{-1} \cap H(\Q)^+$ of $H(\Q)^+$.
  Here, $H(\Q)^+$ denotes the connected component of the identity.
  If we endow $\domain$ with its usual topology and $H(\adeles_f)$ with its adelic topology,
  this becomes a homeomorphism.
\end{lemma}

\subsection{The Weil representation}
In this section we let $G = \SL_2$, viewed as an algebraic group over $\Q$.

Let $\widetilde{G}_\adeles$ denote the $2$--fold
metaplectic cover of $G$ over $\adeles$.
It is a central group extension of $G(\adeles)$ by $\{\pm 1\}$. We write
the elements of $\widetilde{G}_\adeles$ as pairs $(g,\xi)$ with $g \in
G(\adeles)$ and $\xi \in \{\pm 1\}$.

Write $\widetilde{G}_\R$ for
the inverse image of $G(\R) = \SL_2(\R)$ under the covering
map $\widetilde{G}_\adeles \rightarrow G(\adeles)$;
it is thus an extension of $\SL_2(\R)$ by $\{\pm 1\}$.
It is often useful to identify this with
\begin{equation}\label{sqrtmodel}
  \left\{(g,\phi(\tau)) ~|~ g = \twomat abcd \in G(\R),~
    \phi\colon \uhp \to \C\ \text{holomorphic},~ \phi^2(\tau) = c\tau + d\right\},
\end{equation}
endowed with the multiplication $(g_1,\phi_1(\tau))
(g_2,\phi_2(\tau)) = (g_1g_2, \phi_1(g_2\tau)\phi_2(\tau))$.

We write $K=\SL_2(\Zhat) \subset G(\adeles_f)$,
where $\Zhat=\prod_{p<\infty} \Z_p$ and let $\widetilde{K}$
denote the inverse image of $K$ in
$\widetilde{G}_\adeles$.
Similarly, let $K_\infty = \SO(2,\R) \subset G(\R)$
and denote by $\widetilde{K}_\infty$ its full inverse image in
$\widetilde{G}_\R$.
We let $\Gamma = \SL_2(\Z)$ and let $\widetilde{\Gamma}=\Mp_2(\Z)$
denote the inverse image of $\Gamma \subset G(\R)$ inside
$\widetilde{G}_\R$. (Caution: this is not the same as the inverse image of
$\SL_2(\Z)$ as a subgroup of $G(\Q)$ embedded diagonally into
$G(\adeles)$.)

There is a canonical section $s: G(\Q) \to \widetilde{G}_\adeles$.
We write $\widetilde{G}_\Q=s(G(\Q))$ for the image of $G(\Q)$
under the canonical section. With this notation, we have
\[
\widetilde{G}_\adeles = \widetilde{G}_\Q \widetilde{G}_\R \widetilde{K}.
\]
Moreover, we have $s(\Gamma) = \widetilde{G}_\Q \cap \widetilde{G}_\R
\widetilde{K}$.

For $\gamma^\prime \in \widetilde{\Gamma}$, there are unique $\gamma \in
\Gamma, \gamma^{\prime\prime} \in \widetilde{K}$ such that
$s(\gamma) = \gamma'\gamma''$.
The map $\widetilde{\Gamma} \to \widetilde{K}$ defined by $\gamma^\prime \mapsto \gamma^{\prime\prime}$
is a homomorphism.
This homomorphism
gives a representation $\rho$ of $\widetilde{\Gamma}$ on
$S(V(\adeles_f))$ by defining $\rho(\gamma') =
\omega_f(\gamma'')$, where $\omega_f$ is the (finite) Weil
representation. This is sometimes referred to as the Weil
representation of $\widetilde{\Gamma}$.

Let $L \subset V$ be an even lattice, that is the quadratic form
is assumed to be integral on $L$, and let $L'$ be \label{lattice2}
the dual lattice. The finite abelian group $L'/L$ is called the
\emph{discriminant group} of $L$. Given
$\mu \in L'/L$, let $\phi_\mu \in S(V(\adeles_f))$ be the characteristic function
of $\mu + \hat{L}$, where $\hat{L}=L \otimes_\Z \Zhat$.
Let
\[
  A_L = \bigoplus_{\mu \in L'/L} \C\phi_\mu \subset S(V(\adeles_f)).
\]
The representation $\rho$ of $\widetilde{\Gamma}$ acts on $A_L$ because $A_L$
is stable under $\omega_f \mid _{\widetilde{K}}$. Thus, we obtain a finite
dimensional representation $\rho_{L}: \widetilde{\Gamma} \rightarrow \Aut A_L$.
This representation can be described explicitly as follows.
The group
$\widetilde{\Gamma}$ is generated by
\begin{equation}
  S = \left( \Smatrix, \sqrt \tau \right), \quad T = \left( \Tmatrix, 1 \right).
\end{equation}

We have
\begin{equation}
\begin{aligned}
 \rho_L(T)\phi_\mu &= e(Q(\mu))\phi_\mu \\
 \rho_L(S)\phi_\mu &= \frac{e(-\sgn(V)/8)}{\sqrt{|L^\prime/L|}}
\sum_{\nu \in L^\prime/L}e(-(\mu,\nu))\phi_\nu.
\end{aligned}
\end{equation}
Here, $e(x)=e^{2\pi i x}$ and $\sgn(V)$ denotes the signature of $V$, which is
equal to $2-n$ in our case.

\section{Harmonic weak Maa\ss{} forms}
\label{sec:harmonic-weak-maass}
The main reference for this section is the fundamental
article by Bruinier and Funke \cite{brfugeom}.

Let $(V,Q)$ be a rational quadratic space of type $(\bp,\bm)$
and let $L \subset V$ be an even lattice.
Moreover, we denote by $L^-$ the lattice $L$ together with
the quadratic form $-Q$. The associated Weil representation of
$\calL^{-} = (L'/L,-Q)$ is denoted by $\rho_{L^-}$.
It can be viewed as the dual of $\rho_L$.

Let $k \in \frac{1}{2}\Z$.
We define the \emph{Petersson slash operator} on functions
$f: \uhp \rightarrow A_{L}$ by
\begin{equation*}
  \left( f \mid_{k,L} (\gamma,\phi)\right) (\tau) = \phi(\tau)^{-k} \rho_L((\gamma,\phi))^{-1} f(\gamma \tau).
\end{equation*}

A twice continuously differentiable function $f:\uhp \to A_{L}$
is called a {\em harmonic weak Maa\ss{} form} (of weight $k$ with
respect to $\tilde\Gamma$ and $\rho_L$) if it satisfies:
\begin{enumerate}
\item[(i)]
  $f \mid_{k,L} \gamma = f$ for all $\gamma \in \tilde\Gamma$;
\item[(ii)]
  there is a $C>0$ such that $f(\tau)=O(e^{C v})$ as $v\to \infty$
  (uniformly in $u$, where $\tau=u+iv$);
\item[(iii)]
  $\Delta_k f=0$, where
  \begin{align*}
    \Delta_k := -v^2\left( \frac{\partial^2}{\partial u^2}+
      \frac{\partial^2}{\partial v^2}\right) + ikv\left(
      \frac{\partial}{\partial u}+i \frac{\partial}{\partial v}\right)
  \end{align*}
  is the usual weight $k$ hyperbolic Laplace operator.
\end{enumerate}

We denote the space of harmonic weak Maa\ss{} forms of weight $k$ with
respect to $\rho_L$ by $\calH_{k,L}$ and
write $M^!_{k,L}$ for the subspace of weakly holomorphic modular forms.
Moreover, we write $S_{k,L}$ and $M_{k,L}$ for the subspaces of cusp forms
and holomorphic modular forms.

We write the Fourier expansion of an $f \in \calH_{k,L}$ as
\begin{equation}
  f(\tau) = \sum_{\mu \in L'/L} \sum_{n \in \Q} c(n, \mu, v) q^n.
\end{equation}
Since $f$ is harmonic with respect to the weight $k$ Laplace operator,
the coefficients $c(n,\mu,v)$ satisfy $\Delta_k c(n, \mu, v) = 0$.
Computing a basis for the space of solutions to this differential equation,
gives rise to a unique decomposition $f=f^+ + f^-$.

For $k \neq 1$, it is given by
\label{deff}
\begin{align}
  \label{deff+}
  f^+(\tau)&= \sum_{\mu\in L'/L}\sum_{\substack{n\in \Q\\ n\gg-\infty}} c^+(n,\mu) q^n\phi_\mu,\\
  \label{deff-} f^-(\tau) &= \sum_{\mu\in L'/L} \left( c^-(0,\mu)v^{1-k} +  \sum_{\substack{n\in \Q\\
        n \neq 0}} c^-(n,\mu) W(2\pi nv) q^n \right) \phi_\mu,
\end{align}
where $W(a)=W_k(a):= \int_{-2a}^\infty e^{-t}t^{-k}\, dt = \Gamma(1-k,2|a|)$ for $a<0$.

The function $f^+$ is called the \emph{holomorphic part} and
$f^-$ the \emph{non-holomorphic part} of $f$.

For $k=1$, the expansion of the non-holomorphic part is slightly different.
Namely, the two linear independent solutions to the differential equation
$\Delta_1 c(0,v)=0$ are given by $1$ and
$\log(v)$. Therefore, we obtain in this case
\begin{equation}
  f^-(\tau) = \sum_{\mu\in L'/L} \left(
    c^-(0,\mu) \log(v)
    + \sum_{\substack{n\in \Q\\ n \neq 0}} c^-(n,\mu) W(2\pi nv) q^n
  \right) \phi_\mu.
\end{equation}

\begin{definition}
  We define the subspace $H_{k,L} \subset \calH_{k,L}$, consisting of all forms $f \in \calH_{k,L}$, for which
  there is a \emph{Fourier polynomial}
  \[
  P_f(\tau) = \sum_{\mu \in L'/L} \sum_{\substack{n \in \Q \\ n \leq 0}} c^+(n,\mu)e(n\tau) \phi_\mu,
  \]
  such that $f - P_f$ is exponentially decreasing as $v \to \infty$.
\end{definition}
Note that in \cite{brfugeom} this space was denoted $H^{+}_{k,L}$.

There is an antilinear differential operator
$\xi=\xi_k: \mathcal{H}_{k,L} \to M^{!}_{2-k,L^-}$, defined by
\begin{equation}
  \label{defxi} f(\tau)\mapsto \xi(f)(\tau)
  :=v^{k-2} \overline{L_k f(\tau)} = R_{-k} v^k\overline{ f(\tau)}.
\end{equation}
Here $L_k$ and $R_k$ are the Maa\ss{} lowering and raising operators, defined by
\begin{equation*}\label{def:RkLk}
  R_k = 2i \frac{\partial}{\partial\tau} + k v^{-1} \quad
  \text{and} \quad L_k = -2iv^2\frac{\partial}{\partial\overline{\tau}}.
\end{equation*}

The kernel of $\xi$ is equal to $M^!_{k,L}$ and
by \cite[Corollary~3.8]{brfugeom}, the sequences
\begin{gather}
  \label{ex-sequ}
  \xymatrix@1{ 0 \ar[r] & M_{k,L}^! \ar[r] & \calH_{k,L} \ar[r]^-{\xi_{k}} & M^!_{2-k,L^-} \ar[r] & 0} \\
  \xymatrix@1{ 0 \ar[r] & M_{k,L}^! \ar[r] & H_{k,L} \ar[r]^-{\xi_{k}} & S_{2-k,L^-} \ar[r] & 0}
\end{gather}
are exact.

\begin{lemma}[Lemma 3.1 of \cite{brfugeom}]
  The Fourier expansion of $L_k f$ for $f \in \calH_{k,L}$
  is given by
  \begin{equation*}
    L_k f = L_k f^- = -v^{2-k} \sum_{\mu \in L'/L} \left(
      c^-(0,\mu) (k-1 - \delta_{k,1}) +
      \sum_{n\in\Q} c^-(n,\mu)(-4\pi n)^{1-k} e(n\bar{\tau})
    \right) \phi_\mu.
  \end{equation*}
\end{lemma}

\begin{corollary}
  The Fourier expansion of $\xi_k f \in M^!_{2-k,L^-}$ for any $f \in \calH_{k,L}$ is given by
  \begin{equation*}
    -\sum_{\mu \in L'/L} \left(\overline{c^-(0,\mu)} (k -1 - \delta_{k,1})
      + \sum_{n\in\Q} \overline{c^-(-n,\mu)}(4\pi n)^{1-k} e(n\tau)
    \right) \phi_\mu.
  \end{equation*}
\end{corollary}

\begin{lemma}
  If $f \in H_{k,L}$,
  then the non-holomorphic part $f^-$ of $f$ decays exponentially
  as $\Im(\tau) \to \infty$.
\end{lemma}
\begin{proof}
  This follows from the standard growth estimates of the Whittaker functions.
\end{proof}

Using the Petersson scalar product and the operator $\xi$, we obtain a bilinear pairing between $g \in M_{2-k,L^{-}}$
and $f \in H_{k,L}$ via
\begin{equation}
  \label{eq:pairing}
  \{g, f\} := (g, \xi_k(f) )_{2-k} = \int_{\Gamma\bs \uhp} \langle g, \xi_k(f) \rangle v^{2-k} \frac{dudv}{v^2}.
\end{equation}

\begin{lemma}
  Let $f \in H_{k,L}$ and $g \in M_{2-k,L^{-}}$. Then
  \[
    \{ g, f \} = \sum_{\mu \in L'/L} \sum_{n \leq 0} c^{+}(n,\mu)b(-n,\mu),
  \]
  which implies that the pairing only depends on the principal part of $f$ (and on $g$).
  The first exact sequence in \eqref{ex-sequ} implies that the pairing between $S_{2-k,L^{-}}$ and
  $H_{k,L}/M^{!}_{k,L}$ is non-degenerate.
\end{lemma}

\subsection{Principal parts of preimages under $\xi$}
\label{sec:princ-parts-preim}

\begin{theorem}{\cite[Theorem 5.6]{McGrawBasis}}
  \label{thm:McGraw}
  Each of the spaces $M_{2-k,L^{-}}$ and $S_{2-k,L^{-}}$ has a basis of modular forms
  all of whose Fourier expansion have only integer coefficients.
\end{theorem}

\begin{proposition}
  \label{prop:basis-rat-pp}
  The space $H_{k,L}/M^{!}_{k,L}$ has a system of representatives in $H_{k,L}$
  with rational principal parts.
\end{proposition}
\begin{proof}
  Denote the $\Q$-vector space of cusp forms in $S_{2-k,L^{-}}$ with
  rational Fourier coefficients by $S_{2-k,L^{-}}(\Q)$. Moreover,
  we denote by $M^{!}_{k,L}(\Q)$ and $H_{k,L}(\Q)$ the subspaces
  of weakly holomorphic forms and weak Maa\ss{} forms with only
  rational Fourier coefficients in their principal parts.
  The pairing $\{\cdot,\cdot\}$ induces an isomorphism between
  $H_{k,L}(\Q)/M^{!}_{k,L}(\Q)$ and the dual space $S^{*}_{2-k,L^{-}}(\Q)$
  of $S_{2-k,L^{-}}(\Q)$  mapping $\bar{f} \in H_{k,L}(\Q)/M^{!}_{k,L}(\Q)$
  to the linear functional $\{\cdot,f\}$
  for any representative $f \in H_{k,L}(\Q)$ of the class $\bar{f}$.

  By Theorem \ref{thm:McGraw}, there is a basis
  $f_{1},\ldots,f_{r}$ of $S_{2-k,L^{-}}$
  with rational Fourier coefficients.
  Let $\bar{F}_1,\ldots,\bar{F}_r \in H_{k,L}(\Q)/M^{!}_{k,L}(\Q)$
  be the dual basis.
 
  Thus, we have
  \begin{equation*}
    H_{k,L}/M^{!}_{k,L} \cong S^{*}_{2-k,L^{-}} \cong {S^{*}_{2-k,L^{-}}(\Q) \otimes_{\Q} \C}.
  \end{equation*}
  We have seen that the latter space is isomorphic to
  $(H_{k,L}(\Q)/M^{!}_{k,L}(\Q)) \otimes_{\Q}\C$.
\end{proof}

\begin{lemma}
  \label{lem:preim_pp}
  Let $F \subset \C$ be a subfield of $\C$.
  Let $g \in S_{2-k,L^{-}}(F)$ and let $B= \{g_1,\ldots,g_r\}$ be a basis of
  $S_{2-k,L^{-}}$ with all Fourier coefficients contained in $F$.
  Then there is a $\pre{g} \in H_{k,L}$ with $\xi \pre{g}=g$,
  such that the Fourier coefficients of the principal part of
  $\pre{g}$ are contained in the ring $F[S]$ for the
  set $S=\{(g_{i},g_{j})\ \mid\ i,j \in \{1,\ldots,r\} \}$.
\end{lemma}
\begin{proof}
  In the proof of Proposition \ref{prop:basis-rat-pp}, we have seen
  that the dual basis $\calB = \{ G_1,\ldots,G_r\}$
  of $B$ determines a system of representatives for $H_{k,L}/M^{!}_{k,L}$
  with principal parts in $F$.
  For all $i \in \{1,\ldots,r\}$ let $\pre{g}_{i} \in \xi_{k}^{-1}(S_{2-k,L^{-}})$,
  such that $\xi_k \pre{g}_i = g_i$ and write $\pre{g}_{i} = \sum_{j=1}^{r} a_{ij}G_{j}$.
  Since $\calB$ is dual to $B$, we have
  \begin{equation*}
    \{g_{i},\pre{g}_{j}\} = a_{ji}.
  \end{equation*}
  On the other hand, we have $\{g_{i},\pre{g}_{j}\}=(g_i,g_j)$.
  Moreover, if $g \in S_{2-k,L^{-}}(F)$, then $g = \sum_{j=1}^{r} c_{j} g_{j}$ with
  coefficients $c_{j} \in F$. The form $\pre{g} := \sum_{j=1}^{r} c_{j} \pre{g}_{j}$ satisfies
  $\xi_{k} \pre{g} = g$ and the coefficient of index $(\gamma,n)$ is equal to
  $\sum_{j=1}^{r} c_j \sum_{m=1}^{r} a_{jm} G_{m}(\gamma,n)$, where $G_{m}(\gamma,n)$
  denotes the corresponding Fourier coefficient of $G_{m}$.
  Since $G_{m}$ has a only rational coefficients in its principal part,
  the $c_{j}$ are contained in $F$  and $a_{jm} = (g_j,g_m)$, the claim follows.
\end{proof}

\subsection{Operations on vector-valued modular forms and lifts}
Let $M\subset L$ be a sublattice of finite index, then a vector
valued modular form $f \in A_{k,\rho_L}$ can be naturally viewed as
a vector-valued modular form in $A_{k, \rho_M}$. Indeed, we have the
inclusions $M \subset L \subset L'\subset M'$ and therefore
\[
L/M \subset L'/M\subset M'/M.
\]
We have the natural map $L'/M \to L'/L$, $\mu\mapsto \bar \mu$.
\begin{lemma}
  \label{sublattice} There are  two natural maps
  $$
  \res_{L/M}: A_{k,\rho_L} \rightarrow  A_{k,\rho_M},
  \quad f\mapsto f_M
  $$
  and
  $$  \tr_{L/M}: A_{k,\rho_M}\rightarrow  A_{k,\rho_L},
  \quad g \mapsto g^L
  $$
  such that for any $f \in A_{k,\rho_L}$ and $g\in A_{k,\rho_M}$
  $$
  \langle f, \bar g^L\rangle =\langle f_M, \bar g \rangle.
  $$
  They are given as follows. For  $\mu\in M'/M$ and $f \in
  A_{k,\rho_L}$,
  \[
  (f_M)_\mu = \begin{cases} f_{\bar\mu},&\text{if $\mu\in L'/M$,}\\
    0,&\text{if $\mu\notin L'/M$.}
  \end{cases}
  \]
  For any $\bar\mu \in L'/L$, and $g \in A_{k,\rho_M}$, let $\mu$ be a
  fixed preimage of $\bar\mu$ in $L'/M$. Then
  $$
  (g^L)_{\bar\mu} =\sum_{\alpha \in L/M} g_{\alpha +\mu}.
  $$
\end{lemma}

\begin{proof}
  See \cite[Proposition 6.9]{ScheHabil} for the map $\res_{L/M}$.
  The assertion for $\tr_{L/M}$ can be proved analogously.
\end{proof}

\section{Theta functions and regularized theta lifts}
Recall that $G = \SL_2$ and $H = \GSpin(V)$.
Let $\widetilde{G}_\adeles$ denote the $2$-fold
metaplectic cover of $G$ over $\adeles$.
It is a central group extension of $G(\adeles)$ by $\{\pm 1\}$. We write
the elements of $\widetilde{G}_\adeles$ as pairs $(g,\xi)$ with $g \in
G(\adeles)$ and $\xi \in \{\pm 1\}$.

\subsection{Siegel theta functions}
For $g \in\widetilde{G}_\adeles$, $h \in H(\adeles)$ and $\varphi
\in S(V(\adeles))$, we have the usual theta function
\[
\theta(g,h,\varphi) = \sum_{\lambda \in V(\Q)}
(\omega(g,h)\varphi)(\lambda).
\]
This theta function is trivially left invariant under the action of $H(\Q)$
and left invariant under the action of $\widetilde{G}_\adeles$ by Poisson
summation.

The usual vector-valued Siegel theta functions are obtained as follows.
We write $\varphi=\varphi_f \otimes \varphi_\infty$. At the infinite place, we put
$\varphi_\infty(x,z)=e^{-2\pi(Q(x_{z}) - Q(x_\zperp))}$ for an element
$z \in \domain$.

Moreover, to obtain a function on $\uhp$, for $\tau \in \uhp$, we let $g_\tau \in
G(\R)$, such that $g_\tau i = \tau$, for instance
\[
g_\tau = \begin{pmatrix}
  1 & u \\ 0 & 1
\end{pmatrix}
\begin{pmatrix}
  v^{\frac{1}{2}} & 0 \\ 0 & v^{-\frac{1}{2}}
\end{pmatrix}
= \begin{pmatrix}
  v^{\frac{1}{2}} & uv^{-\frac{1}{2}} \\ 0 & v^{-\frac{1}{2}}
\end{pmatrix}
\]
and let $\tilde{g}_\tau=(g_\tau,1) \in \widetilde{G}_\R$ (or rather
$\tilde{g}_\tau=(g_\tau,v^{-\frac{1}{4}})$).

Now, if $L \subset V$ is an even lattice, for $\mu \in L'/L$, we set
$\phi_\mu = \charfn(\mu + L)$, understood as an element of $S(V(\adeles_f))$
via the isomorphism $L'/L \cong \hat{L}'/\hat{L}$, where
$\hat{L}=L \otimes_\Z \Zhat$.

\begin{remark}
   \label{rem:gspin-dg-act}
    Note that $\hat{L} \cap V(\Q)=L$. The group $H(\adeles_f)$ acts on lattices $L
   \subset V$ by $L \mapsto hL := \cap_p (V(\Q) \cap h_pL_p) = V(\Q) \cap
   h\hat{L}$, where $h=(h_p) \in H(\adeles_f)$ and $L_p=L \otimes_\Z \Z_p$.
   Moreover, we have $(hL)^\prime = hL^\prime$ for the dual lattices. Therefore,
   the action of $h$ induces an isomorphism of discriminant groups:
   \begin{equation*}
    \xymatrix@=40pt{
 	L'/L \ar[r] \ar@{=>}[d]^{\rotatebox{270}{$\sim$}} &
   (hL)'/hL \ar@{=>}[d]^{\rotatebox{270}{$\sim$}} \\
 	\hat{L}'/\hat{L} \ar[r] & (h\hat{L})'/h\hat{L}
    }
   \end{equation*}
\end{remark}

We choose a base point $z_0 \in \domain$ and for each $z \in \domain$ an
element $h_z \in H(\R)$, such that $h_z z_0 = z$.
\begin{definition}[Siegel theta function]
  \label{def:sieg-theta-funct}
  \begin{align}
    \theta_\mu(\tau,z,h_f)
    &:=v^\frac{n-2}{4}
    \theta(\tilde{g}_\tau,(h_z,h_f), \phi_\mu \otimes \varphi_\infty(z_0, \cdot)) \\
    &= v^{n/2}\sum_{\lambda \in h_f(\mu + L)}
    \e{Q(\lambda_{z})\,\tau + Q(\lambda_{\zperp})\,\overline{\tau}},
  \end{align}
  where $h_f \in H(\adeles_f)$.
  Using this, obtain a $S(V(\adeles_f))$ valued theta function
  \begin{equation*}
    \Theta_L(\tau,z,h_f) := \sum_{\mu \in L'/L} \theta_\mu(\tau,z,h_f)\phi_\mu.
  \end{equation*}
\end{definition}

\begin{theorem}
  If $K \subset H(\adeles_f)$ is an open compact subgroup that acts trivially on $L'/L$,
  then the Siegel theta function $\Theta_L(\tau,z,h_f)$ is a function on $X_K$.
  Moreover, as a function on $\uhp$, it is a non-holomorphic vector-valued modular form
  of weight $(2-n)/2$, i.e. for
  $\gamma = \left(\smallabcd,\phi(\tau)\right) \in \widetilde{\Gamma}$
  we have:
  \[
  \Theta_L(\gamma\tau, z, h_f) = \phi(\tau)^{2-n}\rho_L(\gamma)
  \Theta_L(\tau,z,h_f).
  \]
\end{theorem}

\subsection{Automorphic Green functions and Borcherds products}
For $f \in H_{k, L}$ with $k = (2-n)/2$, we consider the regularized theta
integral
\[
  \Phi(z,h,f) = \int_{\widetilde{\Gamma}\backslash \uhp}^{\reg} \langle
                 f(\tau), \overline{\Theta_L(\tau,z, h)}\rangle\,v^{\frac{2-n}{2}}d\mu(\tau).
\]

Note that the expression under the integral is
$\tilde{\Gamma}$-invariant.  

We recall the regularization of Harvey and Moore used by Borcherds \cite{boautgra}.
We take the constant term of the
Laurent expansion at $s=0$ of the meromorphic continuation of this limit, i.e.
\[
\int_{\Gamma \backslash \uhp}^{\reg}
\langle f(\tau), \overline{\Theta_L(\tau,z,h)} \rangle d\mu(\tau):= \CT_{s=0} \left[
\lim_{T\rightarrow \infty}
 \int_{\calF_T} \langle f(\tau), \overline{\Theta_L(\tau,z,h)} \rangle v^{-s}
 d\mu(\tau) \right]. 
\]
Assume that $c^+(m, \mu) \in \Z$ for $m<0$.
To $f$ we associate the following divisor on $X_K$
\[
  Z(f) = \sum_\mu \sum_{m>0} c^+(-m,\mu) Z(m,\mu).
\]

For weakly holomorphic $f$, Borcherds proved the following result.
\begin{theorem}[Borcherds, \cite{boautgra}]
Let $f \in M_{(2-n)/2,L}^{!}$ with $c(m,\mu) \in \Z$ for all
$m<0$ and $c(0,0)=0$.
Then there is a function $\Psi$ on $\domain\times H(\adeles_f)$, such that  
\begin{enumerate}
\item $\Psi(z,h,f)$ is a meromorphic modular form for $H(\Q)$ of
  weight $0$ and level $K$, with some unitary multiplier system of
  finite order.
\item $\divisor \bigl( \Psi(z,h,f) \bigr) = Z(f)$.
\item $\Phi(z,h,f) =  -2\log\lvert \Psi(z,h,f) \rvert^{2}$.
\end{enumerate}
\end{theorem}

Extending the space of input functions to harmonic weak Maa\ss{} forms,
Bruinier showed the following result.
\begin{theorem}[Bruinier, \cite{brhabil}]
  \ \\
\begin{enumerate}
\item $\Phi(z,h,f)$ is smooth on $X_K \backslash Z(f)$.\\
  It has a logarithmic
  singularity\footnote{That is,
    for all $z \in X_K$, there is a small neighbourhood $U \subset
X_K$ and a local equation $\divisor g =  -2Z(f)\vert_U$,
with a meromorphic function $g$, such that $\Phi - \log|g|$ extends to a smooth
function on $U$.}
  along $-2 Z(f)$.
  
\item $dd^c \Phi(z,h,f)$ extends to a smooth $(1,1)$ form on $X_K$.

As a current on $X_K$, we have 
\[
dd^c  \left[ \Phi(z,h,f) \right] + \delta_{Z(f)} 
= \left[dd^c  \Phi(z,h,f)\right].
\] 
(I.e. for all $\alpha \in A_{c}^{{n-1},{n-1}}(X_K)$, we have
\[
\int_{X_K}\Phi(z,h,f)  dd^c \alpha + 
\int_{Z(f)} \alpha 
= \int_{X_K} \left(dd^c  \Phi(z,h,f)\right) \alpha.)
\]
\item We have 
\[
\laplace_z \Phi(z,h,f) = \frac{n}{4} c^+(0,0).
\]
Here, $\laplace_z$ denotes the $H(\R)$-invariant Laplacian on $\domain$,
normalized as in \cite{brhabil}.
\end{enumerate}
\end{theorem}

\section{CM cycles and values} \label{sec:cmcycles}
Let $U \subset V$ be a $2$-dimensional, positive definite rational subspace.
It determines a two point subset $\{z_U^\pm\}\subset \domain$
given by $U(\R)$ with the two possible choices of orientation.
Denote by $V_- = U^\perp \subset V$ the $n$-dimensional negative definite orthogonal complement
of $U$ over $\Q$. Then we have a rational splitting
\begin{align}
  \label{split} V = U \oplus V_{-}.
\end{align}

We obtain a cycle $Z(U)_K \subset X_K$, which is called the \emph{CM-cycle} corresponding to $U$.
It is obtained by embedding a Shimura variety associated with $U$ into $X_K$, which is given as follows.
Put $T = \GSpin_U$, which we view as a subgroup of $H$ acting trivially on $V_-$.
The group $K_T = K \cap T(\adeles_f)$ is a compact open subgroup of $T(\adeles_f)$.
Then
\[
Z(U)_K = T(\Q) \backslash \{z_U^\pm\} \times T(\adeles_f) / K_T \hookrightarrow X_K.
\]
Here each point in the cycle is counted with multiplicity
$\frac{2}{w_{K, T}}$, where  $w_{K,T}=\# (T(\mathbb Q) \cap K_T)$.

\subsection{The action of $\GSpin_U(\adeles_f)$}
\label{sec:GspinU}
Recall the short exact sequence of algebraic groups over the rationals
\[
1 \longrightarrow \G_m \longrightarrow T \longrightarrow \SO_U \longrightarrow 1.
\]

Over $\Q$, we have that $C_U^0 \cong \Q(\sqrt{-\abs{\det(U)}})$,
the even part of the Clifford algebra, is isomorphic to
an imaginary quadratic number field, denoted by $k$.
The Clifford norm corresponds to the norm $\N(x)=x\bar{x}$ in $k$.
Here, $\bar{x}$ denotes complex conjugation.
Moreover, the group $SO_U$ is isomorphic to $k^1=\{ x \in k\ \mid\ \N(x)=1\}$
and $\GSpin_U \cong k^{\times}$ is the multiplicative group of $k$.
For more details, we refer to Kitaoka's book \cite{kitqf}
and the chapter by Bruinier in \cite{123}.

The map $T \mapsto \SO_{U}$ is given by $x \mapsto x/\bar{x}$.
This is essentially Hilbert's theorem '90 but can also be seen directly
by a short calculation using the definition of the Clifford group.

We let $\calO \subset k$ be the ring of integers in $k$.
For a place $\frakp$ of $k$ we let $k_{\frakp}$ denote the completion
of $k$ with respect to the valuation $v_{\frakp}$ corresponding to $\frakp$.
Let $\calO_{\frakp} \subset k_{\frakp}$ be its valuation ring.
We consider the adeles over $k$, which is the restricted product
$\adeles_{k} = \prod_{\frakp}^{'} k_{\frakp}$ with respect to $\calO_{\frakp}$.
The finite adeles are denoted by $\adeles_{k,f} = \prod_{\frakp < \infty}^{'} k_{\frakp}$
and we also consider the compact open subgroup
$\hat{\calO} = \prod_{\frakp < \infty} \calO_{\frakp}$.
Moreover, we denote by $\adeles_{k,f}^{\times}$ the ring of ideles over $k$.
We see that $T(\adeles_{f}) \cong \adeles_{k,f}$.
In the following, we will use the convention, that all products
(or intersections etc.) indexed only by $\frakp$ will mean that we only take the finite
places into account. Whenever we want to take the infinite places into account as well,
we will write $\frakp \leq \infty$. The following lemma is well known.

\begin{lemma}
  If $K_T = \hat{\calO}$, then $Z(U)$ is isomorphic to two
  copies of the ideal class group $Cl(k)$ of $k$,
  that is, $Z(U)_K \cong Cl(k) \times \{ z_U^{\pm} \}$.
\end{lemma}
\begin{proof}[Sketch of proof]
  This follows from the fact that the homomorphism
  \[
  \adeles_{k,f}^{\times} \rightarrow Cl(k), \quad (x_{\frakp})_{\frakp} \mapsto [\prod_{\frakp} \frakp^{v_{\frakp}(x_{\frakp})}]
  \]
  is surjective. The group $k^{\times}$ corresponds to the principal ideals and the group $\hat{\calO}$
  maps to the trivial class.
  This leads to an isomorphism
  \[
  k^{\times} \bs \adeles_{k,f}^{\times} / \hat{\calO} \cong Cl(k). \qedhere
  \]
\end{proof}

The group $H(\adeles_{f})=\GSpin_{V}(\adeles_{f})$ acts on lattices in $V$.
If $h = (h_p)_p \in \GSpin_{V}(\adeles_{f})$ and $L = \bigcap_{p} L_{p}$ is a lattice in $V$,
where $L_p = L \otimes_{\Z} \Z_{p}$ is the corresponding $p$-adic lattice, then
$hL = \bigcap_{p} (h_p L_p)_p$.

In the following, we will examine the action of $T(\adeles_{f})$ on
lattices in $U$ more closely.
The $\Q_{p}$ vector space $k \otimes_{\Q} \Q_p$ is in fact an algebra with the multiplication
$(a \otimes b)(c \otimes d) = ac \otimes bd$. It is isomorphic \cite[Satz 8.3]{neukirchalgzt}
to the product
\[
\prod_{\frakp \mid p} k_{\frakp}.
\]
On each of the factors, an element $x_{\frakp} \in k_{\frakp}^{\times}$ acts via $x_{\frakp}/\bar{x}_{\frakp}$
as above in the situation over $\Q$.

For our purposes, we assume that the lattice in $U$ corresponds to a fractional ideal $\fraka \in k$.
Then the action of $x \in T(\adeles_{f})$ is given as follows.
We write $\pi_{\frakp} \in \calO_{\frakp}$ for a uniformizer in $\calO_{\frakp}$.
This means that the only prime ideal in $\calO_{\frakp}$ is generated by $\pi_{\frakp}$
and every element in $k_{\frakp}$ can be written as $\pi^{m}u$, where
$m \in \Z$ and $u \in \calO_{\frakp}^{\times}$. The fractional ideal $\fraka$ of $k$ is
contained in the fractional ideal $(\pi_{\frakp})^{v_{\frakp}(\fraka)}$ of $k_{\frakp}$
and in fact, we have
\[
\fraka = \prod_{\frakp} \frakp^{v_{\frakp}(\fraka)} = \bigcap_{\frakp} (\pi_{\frakp})^{v_{\frakp}(\fraka)}.
\]

Let $h=(h_{\frakp})_{\frakp} \in T(\adeles_f)$. Then we have
\begin{equation}
  \label{eq:ha}
  h\fraka = \prod_{\frakp} \frakp^{v_{\frakp}(\fraka)+\mu_{\frakp}(h)},
\end{equation}
where
\[
\mu_{\frakp}(h) =
\begin{cases}
  0, &\text{if } \frakp = \bar\frakp\\
  v_{\frakp}(h_{\frakp})-v_{\bar\frakp}(h_{\bar\frakp}), &\text{otherwise}.
\end{cases}
\]

This can be seen as follows.
For primes $\frakp$ with $\frakp = \bar\frakp$, i.e. for inert and ramified primes, the action of $T(\Q_{p})$
is the same as above and therefore the valuation at those primes does not change.
For split primes, i.e. primes where $p\calO=\frakp \bar{\frakp}$, the action is different.
Consider a pair $(x,y) \in k_{\frakp}^{\times} \times k_{\bar\frakp}^{\times} \cong T(\Q_p)$.
Then, each of the factors $k_{\frakp}$ and $k_{\bar\frakp}$ is isomorphic to $\Q_{p}$
for the product being a two-dimensional vector space over $\Q_p$. We can write $x=p^{r}u, y=p^{s}v$
for $r,s \in \Z$ and $u,v \in \Z_p^{\times}$.
Then $(x,y)$ acts on a pair $(a,b) \in k_{\frakp} \times k_{\bar\frakp}$
via $(a,b) \mapsto (p^{r-s}uv^{-1}a,p^{s-r}vu^{-1}b)$. Thus, we obtain the formula above.

In particular, if we assume that $\hat\calO \supset K_{T}$, we obtain a surjective map
\[
T(\Q) \bs T(\adeles_f) / K_T \rightarrow Cl(k),\quad h \mapsto [h]
\]
and an action of $\GSpin_U(\adeles_f)$ on the class group.
From the formulas above, we see that this action corresponds to
multiplication by the class $[h]/\overline{[h]}$. Here, $\overline{[h]}$
denotes the complex conjugate class of $[h]$.
Note that $[h]/\overline{[h]} = [h]^2$ since in an imaginary quadratic field
the ideal $\frakp\bar\frakp$ is a principal ideal for all prime ideals $\frakp \subset \calO$.
(It is either generated by $p=\N(\frakp)$ or by $p^{2}$.)

Therefore, the class $[h\fraka] \in Cl(k)$ is given by
\begin{equation}
  \label{eq:classha}
  [h\fraka] = [h]^2 [\fraka],
\end{equation}
in accordance with the fact that $\GSpin_{U}(\adeles_{f})$
acts on lattices in the same genus.

\begin{remark}
  We should warn the reader that if $P=\fraka \subset k$ is a fractional
ideal, the theta function $\Theta_P(\tau,h)$ is in general not equal to
the vector-valued theta function corresponding to $\fraka(h)^{2}\fraka$. This is
due to the fact that $T(\adeles_{f})$ also acts on the components via automorphisms,
see also Remark \ref{rem:gspin-dg-act} and Definition \ref{def:sieg-theta-funct}.
\end{remark}

\subsection{The value of the theta lift at a CM cycle}
\label{sec:value-theta-lift}
We are interested in the value of a theta lift $\Phi(z,f)$ at single CM points,
as well as the evaluation at the CM-cycle $Z(U)=Z(U)_K$.
The latter is defined to be the sum
\[
\Phi(Z(U),f) = \frac{2}{w_{K,T}} \sum_{z \in \supp{Z(U)_{K}}} \Phi(z,f).
\]
This value has been studied by Schofer \cite{schofer} for the Borcherds lift of weakly holomorphic
modular forms and by Bruinier and Yang \cite{bryfaltings} for weak Maa\ss{} forms.
We review their main result.

The splitting \eqref{split} yields two lattices, $P$ and $N$, defined by
\begin{align*}
  N = L \cap V_{-},\quad  P = L \cap U.
\end{align*}
Their direct sum, $P \oplus N \subset L$, is a sublattice of finite index.

For $z=z_U^\pm$ and $h\in T(\adeles_f)$,
the Siegel theta function $\theta_{P \oplus N}(\tau,z,h)$ splits up
as a product
\begin{align}
  \label{splittheta}
  \theta_{P \oplus N}(\tau,z_U^\pm,h)= \Theta_P(\tau,z_{U}^{\pm},h) \otimes \Theta_N(\tau).
\end{align}
Here $\Theta_N(\tau)=\Theta_N(\tau,1)$ is the
$S_N$-valued theta function of weight $n/2$ associated to the
negative definite lattice $N$. Note that $v^{-n/2}\overline{\Theta_N(\tau)}$
is the holomorphic theta function corresponding to the positive definite lattice
$N^{-}$.

Attached to $P$ there is an incoherent Eisenstein series $\hat E_P(\tau,s)$
of weight 1 transforming with the representation $\overline{\rho_{P}}$.
Its central value at $s=0$ vanishes but it is the value
of the derivative $\frac{\partial}{\partial s}\hat E_P(\tau,s)$ at $s=0$
that carries the arithmetical data which contributes to the CM values.
The function $\pre{E_P}(\tau) = \frac{\partial}{\partial s}\hat E_P(\tau,s)\mid_{s=0}$ is a harmonic weak Maa\ss{}
form of weight $1$ with respect to $\overline{\rho_{P}}$.
We denote, as usual, by $\prep{E_{P}}(\tau)$ its holomorphic part.

If $S(q)=\sum_{n\in \Z} a_n q^n$ is a Laurent series in $q$ (or a
holomorphic Fourier series in $\tau$), we write
\begin{align}
  \CT(S)=a_0
\end{align}
for the constant term in the $q$-expansion.

\begin{theorem} \label{thmBrYCM}
  \label{thm:fund} The value of the theta lift
  $\Phi(z,h,f)$ at the CM cycle $Z(U)_{K}$ is given by
  \begin{align*}
    \Phi(Z(U),f)& = \frac{4}{\vol(K_T)}\left( \CT\left(\langle
        f_{P \oplus N}^+(\tau),\, \Theta_{N^{-}}(\tau)\otimes \ \prep{E_{P}}(\tau)\rangle\right) +
      L'(\xi(f), U,0)\right).
  \end{align*}
  Here, $L'(\xi(f),U,s)$ is the derivative with respect
  to $s$ of a certain $L$-series.
\end{theorem}
\begin{proof}
  This is Theorem 4.7 in \cite{bryfaltings}.
\end{proof}

The proof involves the Siegel-Weil formula and the standard Eisenstein
series of weight $1$, defined as
It is defined as
\begin{equation}
  \label{eq:holeis}
  E_{P}(\tau,s)
  = \frac{1}{2} \sum_{\gamma \in \Gamma_{\infty} \bs \Gamma} (\Im(\tau)^s \varphi_0) \mid_{1,P} \gamma.
\end{equation}
The series converges for $\Re(s)>1$ and has a meromorphic continuation to the whole complex $s$-plane.

We normalize the measure on $\SO(V)(\R) \cong \SO(2,\R)$
such that $\vol(\SO(V)(\R))=1$.
This implies that we have $\vol(\SO(V)(\Q)\bs\SO(V)(\adeles_f))=2$.
\begin{theorem}[Siegel-Weil formula]
  \label{sw02} The Eisenstein series $E_{P}(\tau,s)$ is holomorphic at $s=0$ and we have
  \[
  \int_{\SO(U)(\Q)\backslash \SO(U)(\adeles_f)} \Theta_P(\tau,z_U,h_f)\,dh
  = 2 E_P(\tau,0).
  \]
\end{theorem}
\begin{proof}
  This has been proved by Kudla and Rallis \cite{kudla-rallis-sw1}.
  See also Bruinier and Yang \cite[Proposition 2.2]{bryfaltings}.
\end{proof}

\begin{remark}
  Note that the Siegel-Weil formula provides a proof that $E_{P}(\tau,0) \in M_{1,P}$.
  This can be seen even better by applying Proposition \ref{lem:schoferSumInt} in the
  next section.
\end{remark}

For later reference, we write the Fourier expansion of $E_P(\tau,0)$ as
\begin{equation}
  \label{eq:EPFourier}
  E_P(\tau,0) = \phi_{0} + \sum_{\beta \in P'/P} \sum_{n \in Q(\beta) + \Z_{>0}} \rho(n,\beta) e(n\tau) \phi_{\beta}.
\end{equation}
A crucial fact is that $\pre{E_P}(\tau)$ maps to $E_{P}(\tau,0)$ under the $\xi_{1}$ operator.
This has been stated by Bruinier and Yang \cite[Remark 2.4]{bryfaltings}. Note
that with their normalization of the Eisenstein there is a factor $1/2$ missing in the remark.
It follows directly from equation (2.19) in \cite{bryfaltings}.

This identity can also be stated in terms of differential forms.
\begin{lemma}
  \label{eis3} We have
  \[
  -\bar\partial\left(\pre{E_P}(\tau)\, d\tau\right)
  = v \overline{E_P(\tau, 0)}\, d\mu(\tau).
  \]
\end{lemma}
Here $\bar\partial$ denotes the usual Dolbeaut operator such that $d=\partial + \bar\partial$.
Using this lemma, the proof is essentially an application of Stoke's theorem.

The Fourier expansion of $\pre{E_P}(\tau)$ can be determined using a very general result
by Kudla and Yang \cite{KudlaYangEisenstein} on the coefficients of Eisenstein series on $\SL_{2}$.
The following result can be found as Proposition 7.2 in \cite{KudlaYangEisenstein}.

We write
\begin{equation}
  \label{eq:EpreFourier}
  \pre{E_P}(\tau) = \sum_{\beta \in P'/P} \left( \sum_{n \in Q(\beta) + \Z_{>0}}\!\!\! \kappa(n,\beta) e(n\tau)
    + \delta_{0,\beta}\log(v)
    -\!\!\!\!\! \sum_{n \in Q(\beta)+\Z_{< 0}}\!\!\!\!\! \rho(-n,\beta) W(2\pi nv) e(n\tau) \right) \phi_{\beta}.
\end{equation}

\begin{theorem}
  \label{thm:EpreFourier}
  Let $D<0$ be fundamental discriminant and
  $h_k$ be the class number of $k=\mathbb Q(\sqrt D)$.
  Write the global Hilbert symbol $\psi=(D, \cdot)_\adeles=\prod_p \psi_p$
  as a product of local quadratic characters.
  Assume that $(P,Q) = (\fraka,\frac{\N}{\N(\fraka)})$ for a fractional
  ideal $\fraka \subset k$. Moreover, write
  $$
    \Lambda(\chi_{D},s) = \abs{D}^{\frac{s}{2}} \pi^{-\frac{s+1}{2}} \Gamma\left(\frac{s+1}{2}\right) L(\chi_D,s)
  $$
  for the completed $L$-function associated with $\chi_{D}$, such that $\Lambda(1-s,\chi)=\Lambda(s)$.
  For $n>0$, define $\Diff(m)$ to be the set of primes $p < \infty$ such that $\chi_p(-n\N(\fraka))=-1$.
  For $n<0$, let $\Diff(m)$ be the set of such finite primes together with $\infty$.

  Let $\beta \in P'/P$ and $n >0$ such that $n \in Q(\beta) +\Z$.
  Denote by $o(n)$ the set of primes $p$ such that $\ord_{p}(nD)>0$.
  Then $\kappa(n,\beta)=0$ unless $\abs{\Diff(n)}=1$. Assume that $\Diff(n)=\{p\}$.
  \begin{enumerate}
  \item If $p$ is inert in $k_{D}$, then
    \[
      \Lambda(\chi_D,0) \kappa(n, \beta) = -2^{o(n)}(\ord_p(n) +1) \rho(n|D|/p) \log p.
    \]
  \item If $p$ is ramified in $k_{D}$, then
    \[
      \Lambda(\chi_D,0) \kappa(n, \beta) = -2^{o(n)}(\ord_p(nD)) \rho(n|D|) \log p.
    \]
  \end{enumerate}
  Finally, we have for the constant term
  $$
    \kappa(0, 0)= - 2 \frac{\Lambda'( \chi_{D}, 0)}{\Lambda(\chi_{D}, 0)}.
  $$
\end{theorem}

\begin{remark}
  Note that by the class number formula $\Lambda(\chi_{D},0)=\Lambda(\chi_{D},1)=\frac{2}{w_{D}}  h_{D}$.
\end{remark}

\begin{remark}
  If you compare the theorem above with Theorem 4.1 in \cite{schofer}, note that Schofer's
  normalization of the completed $L$-function $\Lambda(\chi_{D},s)$
  does not include the factor $\abs{D}^{s/2}$.
\end{remark}

\section{The value of $\Phi(z,f)$ at a single CM point}
\label{sec:cmvalue-single}
We are now interested in computing the value of the theta lift $\Phi(z,f)$ at a single CM point.
For this, we consider the following finitely branched cover of $Z(U)_K$. We let
$K_{P} \subset K_{T} \subset T(\adeles_f)$ be a compact open subgroup
such that $K_{P}$ preserves $P$ and acts trivially on $P'/P$. Consider the
Shimura variety
\begin{equation*}
  Z(U)_{P,K} = T(\Q) \bs \{ z_U^{\pm} \} \times T(\adeles_f) / K_{P}.
\end{equation*}

This is isomorphic to two identical copies of the ``class group''
\begin{equation*}
  C_{P,K} = T(\Q) \bs T(\adeles_f) / K_{P}
\end{equation*}
and defines a cover of the CM cycle $Z(U)_K$ with $[C_{K} : C_{P,K}]$ branches.

Since $K_{P}$ acts trivially on $P'/P$,
for an element $h \in C_{P}$, the value $\Theta_P(\tau,z_{U}^{\pm},h)$ makes sense
and defines a holomorphic modular form contained in $M_{1,P}$.

As before, we would like to apply Stoke's theorem for the evaluation
at a single CM point. The existence of a preimage of each theta function is clear by
the exact sequence \eqref{ex-sequ}. However, we would like to eliminate the ambiguity in
choosing the preimages as much as possible and also pick particularly ``good'' preimages
with coefficients that are as ``nice'' as possible.
A first step is to fix the principal parts.

Write the
Fourier expansion of $\Theta_P(\tau,z_{U}^{\pm},h)$ as
\begin{equation}
  \label{eq:thPFourier}
  \Theta_P(\tau,z_{U}^{\pm},h) = \sum_{\beta \in P'/P} \sum_{n \geq 0} c_{\beta}(h,n) e(n\tau) \phi_{\beta}.
\end{equation}
Since $P$ is positive definite $c_{\beta}(h,0)=0$ for $\beta \neq 0$ and $c_{0}(h,0)=1$.
The constant coefficient of the holomorphic Eisenstein series
$E_{P}(\tau,0)$ is also equal to 1. Thus, we have the decomposition
\begin{equation}
  \label{eq:thPdec}
  \Theta_P(\tau,z_{U}^{\pm},h) = E_P(\tau,0) + g_{P}(\tau,h),
\end{equation}
where for each $h \in C_{P,K}$ the form $g_{P}(\tau,h) \in S_{1,P}$ is a cusp form of weight 1.

We let $s_1,\ldots,s_r \in S_{1,P}$ be a basis with integral Fourier coefficients.
Moreover, by Lemma \ref{lem:preim_pp}, there are weak Maa\ss{} forms
$S_1,\ldots,S_r \in H^{+}_{1,P}$, such that $\xi_1 S_{i} = s_{i}$ and
each coefficient in the principal part of $S_{i}$ is of the form
$\sum_{j=1}^{r} a_{ij} \cdot (s_i,s_j)$ with $a_{ij} \in \Q$.

Now we can define appropriately normalized preimages of the cusp forms
$g_{P}(\tau,h)$ and the theta functions $\Theta_P(\tau,z_{U}^{\pm},h)$.
We define coefficients $a_{i}(h)$ by $\sum_{i=1}^{r} a_i(h) s_{i} = g_{P}(\tau,h)$.
Using the same set of coefficients, we put $\pre{g_{P}}(\tau,h) := \sum_{i = 1}^{r}a_i(h) S_{i}$.

\begin{remark}
Choosing the preimages this way is unique up to the choice of the representatives
of the dual basis. In this sense, the choice is unique up to an element in
$M_{1,P^{-}}^{!}(\Q)$. In the end, the forms $\pre{g}_P(\tau,h)$ are unique up to
an element in $M_{1,P^{-}}^{!}(R)$, where $R=\Q[S]$ for the set
$S:=\{(s_i,s_j)\ \mid i,j\in \{1,\ldots,r\} \}$.
\end{remark}

The following proposition summarizes our construction.
\begin{proposition} \label{SumG=E}
  For $h \in C_{P,K}$ define
  \[
  \pre{\Theta_{P}}(\tau,h) := \pre{E_P}(\tau) + \pre{g_P}(\tau,h).
  \]
  Then $\xi_1 \pre{\Theta_P}(\tau,h) = \Theta_P(\tau,z_{U}^{\pm},h)$ and
  we have
  \[
  \sum_{h \in C_{P,K}} \pre{\Theta_{P}}(\tau,h)  = \frac{w_{P}}{\vol(K_{P})} \pre{E_P}(\tau)
  \]
  and
  \[
  \frac{w_{P}}{\vol(K_P)} = \abs{C_{P,K}}.
  \]
  Here, $w_{P} = \# (T(\Q) \cap K_{P})$.
\end{proposition}
For the proof, we repeat the following Lemma of Schofer\footnote{Note that the
factor $2/w_{P}$ is missing in \cite{schofer}.} \cite[Lemma 2.13]{schofer}.
\begin{lemma}
  \label{lem:schoferSumInt}
  Let $B(h)$ be a function on $T(\adeles_f)$ depending only on the image
  of $h$ in $\SO(U)(\adeles_f)$. Assume that $B$ is invariant under $K_P$ and $T(\Q)$. Then
  \[
  2 \frac{\vol(K_P)}{w_{P}} \sum_{h \in C_{P,K}} B(h)
  = \int_{\SO(U)(\Q) \bs \SO(U)(\adeles_f)} B(h) dh.
  \]
\end{lemma}

\begin{proof}[Proof of Proposition \ref{SumG=E}]
  Setting $B(h) = \Theta_N(\tau,(z_U^+,h)) = \Theta_N(\tau,(z_U^-,h))$ in the lemma we get
  \[
  \sum_{h \in C_{P,K}} \Theta_P(\tau,(z_U^{\pm},h))
  = \frac{w_{P}}{2\vol(K_P)}
  \int_{\SO(U)(\Q) \bs \SO(U)(\adeles_f)} \Theta_P(\tau,(z_U^{+},h)) dh.
  \]
  The latter integral is equal to $2E_P(\tau,0)$ by the Siegel-Weil formula
  (Theorem \ref{sw02}).
  Therefore, since $\Theta_P(\tau,z_{U}^{\pm},h) = E_P(\tau,0) + g_{P}(\tau,h)$,
  we have indeed
  \begin{equation*}
    \sum_{h \in C_{P,K}} g_{P}(\tau,h) = 0.
  \end{equation*}
  Consequently,
  \begin{equation}
    \label{eq:2}
    \sum_{h \in C_{P,K}} a_{i}(h) = 0
  \end{equation}
  for all $h \in C_{P,K}$, since the $s_{i}$ are linearly independent.
  It follows that
  \begin{align}
    \label{eq:3}
    \xi_1 \pre{g_P}(\tau,h) &=  g_{P}(\tau,h)\ \text{from the definition of }  \pre{g_P}(\tau,h) \text{ and}\\
    \sum_{h \in C_{P,K}} \pre{g_P}(\tau,h) &= 0.
  \end{align}
  Consequently,
  \[
  \sum_{h \in C_{P,K}} \pre{\Theta_{P}}(\tau,h) = \abs{C_{P,K}} \cdot \pre{E_P}(\tau).
  \]
  The second identity follows from Lemma \ref{lem:schoferSumInt} with $B(h)=1$.
\end{proof}

\begin{lemma}
  In terms of differential forms, we have
  \[
  -\bar\partial (\pre{\Theta_P}(\tau,h) d\tau) = v \overline{\Theta_P(\tau,z_{U}^{\pm},h)} d\mu(\tau).
  \]
\end{lemma}

In the following calculation, we assume that $L$ splits as $L=P \oplus N$.
In the end, replacing $f$ by $f_{P\oplus N}$ yields the general result
because
\begin{equation}
  \langle f, \Theta_L \rangle =\langle f_{P\oplus N}, \Theta_P \otimes \Theta_N \rangle,
\end{equation}
since $\theta_{P\oplus N} = \Theta_P \otimes \Theta_N$,
and $\Theta_L = (\theta_{P \oplus N})^L$ by \cite{bryfaltings}.

Following Schofer, we express the theta integral in a way that
is convenient for the following calculations.
\begin{lemma}[Lemma 4.5 of \cite{bryfaltings}]
  \label{lem:philim}
  We have
  \begin{equation*}
    \Phi(z_U^\pm,f) = \lim_{T \to \infty}
    \left( \int_{\calF_T}
      \langle f(\tau), \Theta_{N^{-}}(\tau) \otimes \Theta_P(\tau,z_U^\pm,h) \rangle d\mu(\tau)
      - A_0 \log(T) \right),
  \end{equation*}
  where
  \begin{equation*}
    A_{0} = \CT\left( \langle f^+(\tau),\, \Theta_{N^{-}}(\tau) \otimes \phi_{0+P} \rangle \right).
  \end{equation*}
\end{lemma}

Using the same techniques as Bruinier and Yang \cite{bryfaltings}, we obtain the following
\begin{theorem}\label{thm:value-phiz}
  Let $f \in H_{1-n/2,L}$.
  Then the value of $\Phi(z,h,f)$ for any $(z,h) \in Z(U)$ is given by
  \begin{align*}
    \Phi(z,h,f) &= \CT \left( \langle f^+_{P \oplus N}(\tau), \Theta_{N^{-}}(\tau) \otimes \prep{\Theta_P}(\tau,h) \rangle \right) \\
    &\quad + \int_{\SL_2(\Z) \bs \uhp}^{reg} \langle \xi(f_{P \oplus N}),\Theta_{N^{-}}(\tau) \otimes \pre{\Theta_P}(\tau,h) \rangle d\mu(\tau).
  \end{align*}
\end{theorem}

\begin{remark}
  Note that for $f \in M^{!}_{1-n/2,L}$ the second summand does not occur since $\xi(f)=0$
  for weakly holomorphic $f$.
\end{remark}

\begin{proof}[Proof of Theorem \ref{thm:value-phiz}]
  According to Lemma \ref{lem:philim}, we write
  \begin{equation} \label{eq:philim}
    \Phi(z,f) = \lim_{T \to \infty} \left( I_T(z,f) - A_0 \log(T) \right).
  \end{equation}
  We compute
  \begin{align*}
    I_T(z,f) &= \int_{\calF_T}
    \langle f(\tau), \Theta_{N^-}(\tau) \otimes \overline{\Theta_P(\tau,z_U^\pm,h)} \rangle v d\mu(\tau)\\
    &= - \int_{\calF_T}
    \langle f(\tau), \Theta_{N^-}(\tau) \otimes (\delbar \pre{\Theta_P}(\tau,h) d\tau) \rangle \\
    &= - \int_{\calF_T}
    d \left( \langle f(\tau), \Theta_{N^-}(\tau) \otimes \pre{\Theta_P}(\tau,h) d\tau \rangle \right) \\
    &\quad
    + \int_{\calF_T} \left( \langle \delbar f(\tau), \Theta_{N^-}(\tau) \otimes \pre{\Theta_P}(\tau,h) d\tau \rangle \right).
  \end{align*}
  The second summand contributes the regularized integral to the formula.

  For the first integral, we can apply Stoke's theorem to obtain
  \begin{align*}
    \int_{\calF_T}
    d \left( \langle f(\tau), \Theta_{N^-}(\tau) \otimes \pre{\Theta_P}(\tau,h) d\tau \rangle \right)
    &= \int_{\del \calF_T}
    \langle f(\tau), \Theta_{N^-}(\tau) \otimes \pre{\Theta_P}(\tau,h) d\tau \rangle \\
    &= -\int_{iT}^{iT+1} \langle f(\tau), \Theta_{N^-}(\tau) \otimes \pre{\Theta_P}(\tau,h) d\tau \rangle,
  \end{align*}
  since the integrand is a $\SL_2(\Z)$-invariant differential form and thus the
  equivalent pieces of $\del\calF_T$ cancel.
  We split the integral into three pieces, insert
  this splitting into \eqref{eq:philim} and regroup to obtain the following
  expression
  \begin{gather}
    \lim_{T \to \infty}
    \int_{iT}^{iT+1}  \langle f^+(\tau), \Theta_{N^-}(\tau) \otimes\pre{\Theta_P}^{+}(\tau,h)  d\tau \rangle \label{int1}\\
    + \lim_{T \to \infty}
    \left( \int_{iT}^{iT+1}  \langle f^+(\tau), \Theta_{N^-}(\tau) \otimes \pre{\Theta_P}^{-}(\tau,h) d\tau \rangle
      - A_0 \log(T) \right) \label{int2}\\
    + \lim_{T \to \infty}
    \int_{iT}^{iT+1} \langle f^-(\tau), \Theta_{N^-}(\tau) \otimes \pre{\Theta_P}(\tau,h) d\tau \rangle \label{int3}\\
    + \lim_{T \to \infty} \int_{\calF_T}
    \left( \langle \delbar f(\tau), \Theta_{N^-}(\tau) \otimes \pre{\Theta_P}(\tau,h) d\tau \rangle \right) \label{int4}.
  \end{gather}
  First note that for two vector-valued forms $g,h$ transforming
  with representations $\rho$ and $\bar{\rho}$ the $q$-expansion of $\langle g,h \rangle$
  has integral exponents (the exponents satisfy
  $n - Q(\mu) \in \Z$ and $m + Q(\mu) \in \Z$ which yields
  $n+m \in \Z$).

  Each of the limits above exist.
  The limit in \eqref{int3} is equal to zero due to the exponential decay
  of $f^-(\tau)$: the integral picks out the term for ``$n=0$''
  in the $q$-expansion of the expression, which only picks up coefficients
  of positive index from $\Theta_{N^-}(\tau) \otimes \pre{\Theta_P}(\tau,h)$
  since all coefficients of $f^{-}$ of positive index vanish.

  Moreover, since $\pre{\theta}_{P}^-(\tau,h)$ has an expansion of the
  form $\log(v)\phi_{0+P} + \pre{\theta}_{P}^{--}(\tau,h)$, where $\pre{\theta}_{P}^{--}(\tau,h)$
  decays exponentially as $v \to \infty$, the limit in \eqref{int2} is equal to zero, as well.

  Finally, \eqref{int1} is the constant term and
  \eqref{int4} is the regularized integral in the statement of the theorem.
\end{proof}

Comparing our theorem with the result by Schofer (Theorem \ref{thmBrYCM} for
weakly holomorphic $f$), we obtain the following identity.
\begin{corollary}
  Let $f \in M^!_{1-n/2,L}$ be a weakly holomorphic modular form. Then,
  \begin{align*}
    \Phi(Z(U)_{K_{P}},f) &= \frac{2}{w_{P}} \sum_{(z,h) \in \supp(Z(U))}
    \CT \left( \langle f_{P \oplus N}(\tau),\,
      \Theta_{N^-}(\tau) \otimes \prep{\Theta_P}(\tau,h) \rangle \right) \\
    &= \frac{\deg(Z(U))}{2}
    \left( \CT\left( \langle f_{P \oplus N}(\tau),\,
        \Theta_{N^-}(\tau)\otimes \pre{E}^{+}_P(\tau) \rangle \right) \right).
  \end{align*}
\end{corollary}

\section{The coefficients of the holomorphic part}
\label{sec:coeff-holom-part}
In this section, we want to give a proof of Theorem \ref{thm:coeffs-intro}.
We will prove a slightly more general statement in Theorem \ref{thm:pos-coeffs}.

Consider a basis of $S_{1,P^{-}}$ given by $\{g_{j}\ \mid\ j=1,\ldots,\dim S_{1,P^{-}}=d\}$
that we normalized in the following way. We require that each element $g_{j}$ has only rational
Fourier coefficients and that
\[
  g_j(\tau) = \sum_{m \in \Q} \left(\sum_{\beta \in P'/P} a_{j}(m,\beta) \phi_{\beta} \right) q^{m}
  =  q^{n_{j}} v_{j} + O(q^{n_j +1})
\]
with $v_{j} \in A_{L}$ and $n_{1} \leq n_2 \leq \ldots \leq n_{d}$.
We require that $v_{j} \neq 0$ and that it is of the form
$$
  v_{j} = \phi_{\beta_{j}} + \phi_{-\beta_{j}} + \sum_{\beta \neq \pm \beta_{j}}a_{j}(m,\beta)\phi_{\beta}.
$$
for some $\beta_{j} \in P'/P$ and $a_{i}(n_{j},\pm\beta_{j})= 0$
for any other basis element $g_{i}$ with $i \neq j$.

Moreover, for $0 < m \in \Q$ and $\beta \in P'/P$ with $m \equiv Q(\beta) \bmod \Z$
and $(m,\beta) \neq (n_{j},\beta_{j})$ for all $j$,
there is a weakly holomorphic modular form $f_{\beta,m} \in M_{1,P}^{!}$ with principal part
\[
  q^{-m}(\phi_{\beta}+\phi_{-\beta}) - \sum_{j=1}^{d} a_{j}(m,\beta) q^{-n_{j}}(\phi_{\beta_{j}}+\phi_{-\beta_{j}})
\]
constant term $c_{f_{\beta,m}}(0,0)=0$ and rational Fourier coefficients.
This follows from the construction of our basis,
the exact sequence \eqref{ex-sequ} and Theorem \ref{thm:McGraw}.
Moreover, $c_{f_{\beta,m}}(0,0)=0$ can be achieved by subtracting a multiple of $E_{P}$.

As in the introduction, we consider the theta lift for $P$, that is
\[
  \Phi_{P}(h,f) = \int_{\SL_2(\Z) \bs \uhp}^{\reg} \langle f(\tau), \Theta_P(\tau,h) \rangle v d\mu(\tau),
\]
where $f \in M_{1,P}^{!}$ and $h \in C_{P,K}$.
Moreover, we write $U=P \otimes \Q$ and $T=\GSpin_{U}$, as before, and
$K \subset T(\adeles_{f})$ is an open and compact subgroup, acting trivially on $P'/P$.
Note that we omitted the variable $z \in \domain$ in this case
since the Grassmannian only consists of two points $z_{P}^{\pm}$ giving the same value.

At the same time, we consider the theta lift $\Phi_{P \oplus N}$
where $N$ is an even, negative definite lattice of rank $r$.
Associated with $P$ is a point $z_{P} \in \Gr(V)$ for $V=(P \oplus N) \otimes \Q$.

We have the following seesaw identity (see Theorem 2 of \cite{marynaCMvals})
\[
  \Phi_{P}(h,T_{P \oplus N,P}(g)) = \Phi_{P \oplus N}((z_{P},h),g)
\]
for every $g \in M_{1/2,P \oplus N}$ and $h \in T(\adeles_{f})$.
Here, $T_{P \oplus N,P}: M_{1-r/2,P \oplus N}^{!} \rightarrow M_{1,P}^{!}$
is defined as follows. $T_{P \oplus N,P}(f)(\tau)=\theta_{P \oplus N,P}(\tau) f(\tau)$, where
$\theta_{P \oplus N,P}$ is a matrix given by
\[
  \theta_{P \oplus N,P}(\tau) = \left( \sum_{\substack{n \in N' \\ n + \beta - \lambda \in P \oplus N}}
                             e\left( -Q(n)\tau \right) \right)_{\lambda \in (P \oplus N)'/(P \oplus N),\, \beta \in P'/P}.
\]
We will also make use of the following theorem.
\begin{theorem}
  \label{thm:maryna-embedding}
  Let $f \in M_{1,P}^{!}$ with $c_{f}(0,0)=0$ and integral principal part. Then there is a lattice
  $M(P) \cong P \oplus \Z \lambda$ of type $(2,1)$ and a weakly holomorphic $g_{f} \in M_{1/2,M(P)}^{!}$
  with integral principal part and $c_{g}(0,0)=0$, such that
  \begin{enumerate}
  \item $M(P)$ contains a norm 0 vector,
  \item we have $T_{M(P),P}(g_{f}) = c f$ for some $c \in \Z$,
  \item the CM point $z_{P}$ corresponding to $P \otimes \Q$ in $\Gr(M(P))$ is not contained in $Z(g_{f})$.
  \end{enumerate}
\end{theorem}
\begin{proof}
  This is a variant of Borcherds embedding trick \cite[Lemma 8.1]{boautgra}.
  The main idea is to take two negative definite
  even unimodular lattices $M \subset M_{1},M_{2} \subset L$,
  such that $\Theta_{M_{1}}-\Theta_{M_{2}} = c \Delta$,
  where $\Delta$ is the unique normalized cusp form of weight 12
  for $\SL_{2}(\Z)$.
  This is true for $M_{1} = E_{8}^{3}, M_{2} = \Lambda_{24}$,
  where $E_{8}$ denotes the unique even modular positive definite
  lattice of dimension $8$ and $\Lambda_{24}$ is the Leech lattice.
  In this case, we have $c=720$.
  We consider both lattices with the negative of their quadratic form.
  Then consider $(f/\Delta) \in M_{-11,P \oplus M_{1}}^{!}, M_{-11, P \oplus M_{2}}^{!}$
  and take the difference of the restrictions
  $\res_{(P \oplus M_{1})/(P \oplus M)}(f_{\Delta}) - \res_{(P \oplus M_{2})/(P \oplus M)}(f_{\Delta})$.
  Applying the operator $T_{P \oplus M, P}$ to this function yields $720 f$.
  The 3-dimensional lattice is obtained by picking an appropriate vector $\lambda \in M$,
  perpendicular to $P$, such that $z_{P}$ does not lie in the singular locus of $\Phi_{P \oplus \Z{\lambda}}$.
  For details, we refer to Theorem 5 in \cite{marynaCMvals}.
\end{proof}

\begin{theorem}
\label{thm:pos-coeffs}
Let $P$ be a two-dimensional positive definite even lattice.
We denote by $\Theta_P(\tau,h) \in M_{1,P}$ the corresponding theta function.
For every $h \in C_{P,K}$ there is a harmonic weak Maa\ss{} form
$\pre{\Theta_P}(\tau,h) \in \calH_{1,P^{-}}$ with holomorphic part
\[
  \prep{\Theta_P}(\tau,h) = \sum_{\beta \in P'/P} \sum_{n \gg -\infty} c_{P,h}^{+}(n,\beta) e(n \tau) \phi_{\beta}
\]
satisfying the properties of Proposition \ref{SumG=E}, i.e.
$\xi(\pre{\Theta_P}(\tau,h))=\Theta_P(\tau,h)$ and
\[
  \sum_{h \in C_{P,K}}\pre{\Theta_P}(\tau,h) = \abs{C_{P,K}} \pre{E_P}(\tau)
\]
such that for all $\beta \in P'/P$ and all $m > 0$, we have
\[
  c^{+}_{P,h}(\beta,m) = \log\abs{\alpha(h,m,\beta)} - \sum_{j=1}^{d}a_{j}(\beta,m)\kappa(n_{j},\beta_{j}) - R(h,m,\beta)
\]
for some $\alpha(h,m,\beta) \in \overline{\Q}$.
Here,
\[
  R(h,m,\beta) = \frac{1}{2} \sum_{\gamma \in P'/P} \sum_{n \geq 0} c_{f_{\beta,m}}(n,\gamma) c_{P,h}^{+}(-n,\gamma)
\]
and the coefficients $\kappa(n,\beta)$ are defined in equation \eqref{eq:EpreFourier}.
\end{theorem}
\begin{proof}
  Let $\pre{\Theta_P}(\tau,h) \in \calH_{1,P^{-}}$ be preimages under $\xi$ of the theta functions
  $\Theta_P(\tau,h)$ as in Proposition 8.1.
  In particular, we will write
  $\pre{\Theta_P}(\tau,h) = \pre{E_P}(\tau) + \pre{g_{P}}(\tau,h)$, where $\pre{g_{P}}(\tau,h) \in H_{1,P^{-}}$.
  This implies that $\xi(\pre{g_P}(\tau,h)) \in S_{1,P}$.
  Write
  \[
    \prep{g_{P}}(\tau,h) = \sum_{\beta \in P'/P} \sum_{n \gg \infty} a(h,n,\beta)q^{n}.
  \]
  We can choose $\pre{\Theta_P}(\tau,h)$, such that $a(h,n_{j},\pm\beta_{j})=0$ for all $j=1,\ldots,d$.
  This can be done by replacing $\pre{\Theta_P}(\tau,h)$ by
  \[
    \pre{\Theta_P}(\tau,h) - \sum_{j=1}^{d}a(h,n_{j},\beta_{j})g_{j}
  \]
  which preserves the properties of $\pre{\Theta_P}(\tau,h)$ we have established so far.

  On the one hand, by Theorem \ref{thm:value-phiz}, we get
  \[
    \Phi_P(h,f_{\beta,m}) = 2c^{+}_{P,h}(\beta,m) + 2\sum_{j=1}^{d}a_{j}(\beta,m)\kappa(n_{j},\beta_{j}) + R(h,m,\beta).
  \]
  On the other hand, by the seesaw identity and Theorem \ref{thm:maryna-embedding}, we have
  \[
    \tilde{c}\, \Phi_P(h,f_{\beta,m}) = \Phi_{M(P)}((z_{P},h),g_{f_{\beta,m}}).
  \]
  Here, we have multiplied $f_{\beta,m}$ by a constant in order to obtain an integral principal part.
  Then, by Theorem 13.3 of \cite{boautgra}, we have $\Phi_{M(P)}((z_{P},h),g_{f_{\beta,m}}) = -4 \log\abs{\Psi(z_P,h)}$
  for a meromorphic function $\Psi(z,h)$, defined over $\Q$, on the Shimura variety $X(P)_{K}$ associated with $M(P)$
  for $K = H(\adeles_f)_{M(P)}$.
  Therefore, by CM theory \cite{shimauto}, the value $\Psi(z_P,h)$ is algebraic and this concludes the proof.
  We remark that in our case, for fixed $h$, the function $\Psi(z,h)$
  is in fact a meromorphic modular function for a congruence subgroup $\Gamma \subset \SL_{2}(\Z)$
  with only rational coefficients in its Fourier expansion.
\end{proof}

\subsection{Proof of Theorem \ref{thm:coeffs-intro}}
\label{sec:proof-thm2}
In this section, we will prove Theorem \ref{thm:coeffs-intro}.
For this, we will employ a lifting (or symmetrization)
of scalar-valued modular forms to vector-valued modular forms.
We summarize the important facts that we will need later.

Let $L$ be an even lattice of type $(2,n)$,
level $N$ and determinant $D$.
The group $\Gamma_0(N)$ acts on $\phi_{0}$ via the Weil
representation $\rho_{L}$ by a character. It is given by
\[
  \chi_{L} \left(
  \begin{pmatrix}
    a & b \\ c & d
  \end{pmatrix}
  \right) =
      \begin{cases}
        \left( \frac{(-1)^{\frac{n+2}{2}}D}{d} \right)& \text{if } d>0,\\
        (-1)^{\frac{n+2}{2}} \left( \frac{(-1)^{\frac{n+2}{2}}D}{-d} \right)& \text{if } d<0.
      \end{cases}
\]
Suppose that $N$ is square-free. Then $2+n$ is even and the character is quadratic.
Moreover, this implies that for any $f \in M_{k,L}$, the component function
$f_{0}$ is a modular form in $M_{k}(\Gamma_0(N),\chi_{L})$. Conversely,
we can lift any $f \in M_{k}(\Gamma_0(N),\chi_{L})$ to a vector-valued
modular form by defining
\begin{equation}
  \label{liftdef}
  \S_{L}(f) = \sum_{\gamma \in \Gamma_{0}(N) \bs \SL_{2}(\Z)} (f \mid_{k} \gamma) \rho_{L}(\gamma^{-1})\phi_{0} \in M_{k,L}.
\end{equation}
This construction preserves any analytic properties we might impose. In particular,
the lift also works for weak Maa\ss{} forms and weakly holomorphic modular forms.

Following Bundschuh, we define a subspace of the newforms in $S_{k}(\Gamma_0(N),\chi_{L})$ in the following way.
Let $A=L'/L$ and denote by $A_{p}$ its $p$-component. Moreover, write $\chi_{L}=\prod_{p \mid N} \chi_{L,p}$
as a product of characters modulo $p$ for $p \mid N$.
For each prime $p_{i}$ dividing the square-free level $N=p_{1} \ldots p_{r}$, we define a sign $\varepsilon_{i}$.
\begin{definition}
  If $\dim_{\F_{p_{i}}} A_{p_{i}} \geq 2$ or $p_{i}=2$, we define $\varepsilon_{i}=0$.

  If $\dim_{\F_{p_{i}}} A_{p_{i}} = 1$, $p_{i} \neq 2$ and $N Q \mid_{A_{p_{i}}}$ represents the squares modulo $p_{i}$,
  we define $\varepsilon_{i} = 1$.

  Otherwise, we define $\varepsilon_{i}=-1$.
  Using these signs, we let
  \[
    M_{k}^{\varepsilon_{1},\ldots,\varepsilon_{r}}(N,\chi_{L})
    = \{ f \in S_{k}(\Gamma_0(N),\chi_{L})\, \mid\, c_{f}(n)\neq 0 \text{ implies } \chi_{L,p_{i}}(n) = \varepsilon_{i} \}.
  \]
\end{definition}

\begin{theorem}
  \label{thm:Lift0}
  Let $L$ be an even lattice of square-free level $N$
  and let $f \in M_{k}^{\varepsilon_{1},\ldots,\varepsilon_{r}}(N,\chi_{L})$.
  Assume that $\dim_{\F_{p_{i}}} A_{p_{i}} = 1$ or even for all odd $p_{i}$.
  Then the $0$-component of of $\S_{L}(f)$ is a rational multiple of $f$.
  Namely,
  $$\S_{L}(f)_{0} = \nu \frac{N}{\abs{L'/L}} f.$$
  Here, $\nu = \#\{\mu \in L'/L\, \mid\, Q(\mu) \equiv m \bmod \Z \}$ for any $m \in \Z$ with
  $(m,N)=1$.
\end{theorem}
\begin{proof}
  This follows from the proof (the author only states this for cusp forms but this is not essential)
  of Satz 4.3.9 in \cite{BundschuhDiss}.
  It can be seen by calculating the Fourier expansion
  in terms of the Atkin-Lehner involutions and then using the fact that
  a newform is determined by the coefficients coprime to the level.
\end{proof}
There is also a map that is adjoint to the symmetrization with respect
to the Petersson scalar product. It is simply given by the map $f \mapsto f_{0}$.
\begin{proposition}
  \label{prop:pet-lift}
  Let $f \in S_{k}(\Gamma_0(N),\chi_L)$ and let $F \in M_{k,L}$. Then, we have
  for the Petersson scalar product
  \[
    (\S_{L}(f),F) = (f,F_{0})_{\Gamma_0(N)}.
  \]
\end{proposition}
\begin{proof}
  This can be proved by the usual unfolding method.
\end{proof}

The group $\Og(L'/L)$ acts on vector valued modular forms
by permuting the characteristic functions. That is,
$\sigma \in \Og(L'/L)$ acts via $\phi_{\mu} \mapsto \phi_{\sigma \mu}$.
We need another fact about discriminant forms of square-free level.
\begin{proposition}
  \label{prop:Scheit-sqf}
  Let $L$ be an even lattice of square-free level $N$.
  The orthogonal group $\Og(L'/L)$ acts transitively on elements of the same norm in $L'/L$.
  Moreover, if $F \in M_{k,L}$ is invariant under $\Og(L'/L)$ and $F_{0}=0$, then $F=0$.
\end{proposition}
\begin{proof}
  This is contained in Propositions 5.1 and 5.3 of \cite{Scheithauer-liftings}.
\end{proof}

Recall the setup we used in Section \ref{sec:cmvalue-single}.
We assume that $K \cong \hat\calO_{D}^{*}$ which implies that $C_{P,K} \cong \Cl(D)$.

As before, let $D<0$ denote a fundamental discriminant and let $k_{D} = \Q(\sqrt{D})$ be the
imaginary quadratic field of discriminant $D$. We write $\calO_D$ for the ring of integers in
$k_D$ and $\Cl(D)$ for the ideal class group of $k_D$.
\begin{lemma}
  \label{lem:thetas}
  Let $\fraka \subset \calO_{D}$ be an ideal and let $P=\left(\fraka, \frac{\N(x)}{\N(\fraka)}\right)$
  denote the corresponding even quadratic lattice.
  Let
  \[
    \theta_{\fraka}(\tau) = \sum_{x \in \fraka} e\left( \frac{\N(x)}{\N(\fraka)}\tau  \right)
                         = \sum_{n \in \Z_{\geq 0}} \rho_{\fraka}(n)q^{n}  \in M_1(\Gamma_0(\abs{D},\chi_{D})
  \]
  and
  \[
   \Theta_P^{\sym}(\tau)  = \frac{1}{\abs{\Og(P'/P)}} \sum_{\sigma \in \Og(P'/P)} \Theta_P^{\sigma}(\tau)  \in M_{1,P}.
  \]
  For $h \in C_{P,K}$ corresponding to the ideal class of $\frakb \subset \calO_{D}$, we have
  \[
    \S_{P}(\theta_{\fraka\frakb^{2}})(\tau) = \nu \Theta_{P}^{\sym}(\tau,h),
  \]
  where $\S_{P}$ is the lift defined in \eqref{liftdef}
  and $\nu$ is a positive integer, defined in Proposition \ref{prop:pet-lift}.
\end{lemma}
\begin{proof}
  Note that in our case the level $N$ is equal to $D$, the order of the discriminant group.
  Then the lemma follows from Propositions \ref{thm:Lift0} and \ref{prop:Scheit-sqf}
  since $\Theta_P^{\sym}(\tau)$ is invariant under $\Og(L'/L)$.
\end{proof}

The vector-valued theta functions $\Theta_P(\tau,h)$ correspond to the theta functions in the genus
of a given class. By the genus of the class $[\fraka]$ we mean the set
$\{[\fraka][\frakb]^{2}\, \mid\, \frakb \in \Cl(D)\}$.
Equivalently, this is the subset of $\Cl(D)$ on which the quadratic characters (the genus characters)
have the same values as on $[\fraka]$. This motivates the following definition.
\begin{definition}
We define $\Theta^{\sym}(P)=\langle\, \S_{P}(\theta_{\fraka\frakb^{2}})\, \mid\, \frakb \in \Cl(D) \rangle_{\C}$
to be the span of the lifts of the scalar-valued theta functions in the genus of $\fraka$.
\end{definition}

\begin{remark}
  It follows from Lemma \ref{lem:thetas} that the space $\Theta^{\sym}(P)$ is the space
  spanned by the symmetrized theta functions $\Theta_{P}^{\sym}(\tau,h)$.
\end{remark}
We denote by $\calC$ the group of class group characters.
It is well known \cite{kani-theta,KaniCM} that for $\psi \in \calC$ the form
\[
  \theta_{\psi}(\tau) = \frac{1}{w_{D}} \sum_{[\fraka] \in \Cl(D)} \psi(\fraka) \theta_{[\fraka]}(\tau)
\]
is an Eisenstein series if $\psi^{2}=1$ and a cuspidal normalized newform if $\psi^{2} \neq 1$.
The set $\{ \theta_{\psi}\ \mid\ \psi\in \overline{\calC} \}$, where $\overline{\calC}$ denotes
$\calC$ modulo the relation $\psi \mapsto \overline{\psi}$, is an orthogonal basis of the space of
scalar-valued theta functions $\Theta(D)$.

\begin{definition}
  Let $\psi \in \calC$. We let
  \[
    \Theta_{P}^{\psi}(\tau) = \sum_{h \in C_{P,K}} \psi(h) \Theta_{P}^{\sym}(\tau,h).
  \]
\end{definition}
An easy calculation shows that $\Theta_{P}^{\sym}(\tau)_{0} = \sum_{\frakb \in \Cl(D)} \psi(\frakb) \theta_{\fraka\frakb^{2}}
= w_{D} \sum_{\chi^{2}=\psi} \chi(\fraka)\theta_{\chi}$. Using Lemma \ref{lem:thetas}, we see that $\Theta_{P}^{\psi}$
is in fact a multiple of the lift of this scalar-valued modular form. Moreover, we have the relation
$\bar\psi(\fraka)\Theta_{P}^{\psi} = \overline{\Theta_{P}^{\psi}} = \Theta_{P}^{\bar\psi}$.
Therefore, we can make a choice of representatives
for $\bar\calC^{2}$ to obtain a basis of the space of symmetric theta functions.
Their $\C$-span is the full space $\Theta^{\sym}(P)$, as the next proposition shows.

\begin{proposition}
  Assume that $D \equiv 1 \bmod 4$.
  Then the set $\calB(P) = \{ \Theta_{P}^{\psi}\ \mid\ \psi \in \calC^{2} \}$ spans
  $\Theta^{\sym}(P) \subset M_{1,P}$.
  The elements of $\calB(P)$ are permuted by the action of $\Gal(\C/\Q)$.
  Moreover, $(\Theta_{P}^{\psi},\Theta_{P}^{\chi})=0$ unless $\psi=\chi$ or $\psi=\overline{\chi}$.
\end{proposition}
\begin{proof}
  The orthogonality and the Galois action follows from the fact that we identified
  $\Theta_{P}^{\psi}$ as the lift of $w_{D} \sum_{\chi^{2}=\psi} \chi(\fraka)\theta_{\chi}$
  and Proposition \ref{prop:pet-lift}.
  That this set spans $\Theta^{\sym}(P)$ follows from the calculation
  \[
    \theta_{\fraka\frakb^{2}} (\tau) = \sum_{\chi \in \bar\calC} (\chi(\fraka\frakb^{2}) + \overline{\chi(\fraka\frakb^{2})}) \frac{w(\chi)}{2}
                                 \theta_{\chi}(\tau)
                                 = \sum_{\psi \in \bar\calC^{2}} (\psi(\frakb) + \bar\psi(\frakb)\bar\psi(\fraka))
                                   \frac{w(\psi)}{2} \sum_{\chi^{2}=\psi} \chi(\fraka) \theta_{\chi}(\tau).
  \]
  Here, $w(\psi)=\#\{ \psi, \bar\psi \}$. Note that this is the same as $w(\chi)$ for any $\chi \in \calC$
  with $\chi^{2}=\psi$ since $D$ is odd.
\end{proof}

\begin{lemma}
  \label{lem:petersson-quot}
  Let $F \subset \C$ be a number field and let $f \in S_{1,P}(F)$.
  Then, we have
  \begin{equation}
    \label{eq:petrat}
    \alpha_{f,\psi} = \frac{(f,\Theta_{P}^{\psi})}{(\Theta_{P}^{\psi},\Theta_{P}^{\psi})} \in F \cdot K_{\psi},
  \end{equation}
  where $K_{\psi} \subset \C$ is the field containing the Fourier coefficients of $\Theta_{P}^{\psi}$.
  Moreover, for $\sigma \in \Gal(\C/\Q)$, we have
  \begin{equation}
    \label{eq:petratgal}
    \alpha_{f,\psi}^{\sigma}
                    = \frac{(f^{\sigma},(\Theta_{P}^{\psi})^{\sigma})}{((\Theta_{P}^{\psi}))^{\sigma},(\Theta_{P}^{\psi})^{\sigma})}
                    = \alpha_{f^{\sigma},\psi^{\sigma}}.
  \end{equation}
\end{lemma}
\begin{proof}
  This is a standard argument for newforms. We adapt the proof of Shimura \cite[Lemma 4]{shimura-zeta-cusp}.
  We write $f=\sum_{\psi \in \calC^{2}} a_{f,\psi} \Theta_P^{\psi} + g$, where $g \in \Theta(P)^{\perp}$
  and the coefficients $a_{f,\psi}$ are some complex numbers.

  Then $(f,\Theta_{P}^{\psi}) = \alpha_{f,\psi} (\Theta_P^{\psi},\Theta_P^{\psi})$,
  with $\alpha_{f,\psi}=a_{f,\psi} + a_{f,\overline{\psi}}\overline{\psi}(\fraka)$.
  Therefore,
  \[
    (f^{\sigma},(\Theta_P^{\psi})^{\sigma}) = (f^{\sigma},\Theta_P^{\psi^{\sigma}})
      = \alpha_{f^{\sigma},\psi^{\sigma}} ((\Theta_P^{\psi})^{\sigma},(\Theta_P^{\psi})^{\sigma}).
  \]
  However, since
  \[
    f^{\sigma} = \sum_{\psi \in \calC^{2}} a_{f,\psi}^{\sigma} (\Theta_P^{\psi})^{\sigma} + g^{\sigma}
              = \sum_{\psi \in \calC^{2}} a_{f,\psi}^{\sigma}\, \Theta_P^{\psi^{\sigma}} + g^{\sigma},
  \]
  we have $\alpha_{f,\psi}^{\sigma} = \alpha_{f^{\sigma},\psi^{\sigma}}$.
  This yields the formula for $\alpha_{f,\psi}^{\sigma}$, which also implies the first statement.
\end{proof}

\begin{lemma}
  \label{lem:orthrat}
  The orthogonal complement of $\Theta^{\sym}(P)$ with respect to the Petersson inner product
  has a basis with only rational Fourier coefficients.
\end{lemma}
\begin{proof}
  The space $\Theta^{\sym}(P)$ has a basis with only rational Fourier coefficients
  given by $\Theta_{P}^{\sym}(\tau,h)=\S_{P}(\theta_{[\fraka][\fraka(h)]^{2}}(\tau))$, where $h$ runs through $C_{P,K}$.
  We can complete it to a basis of $M_{1,P}(\Q)$, not necessarily orthogonal.
  Denote the remaining basis elements by $\{f_1\, \ldots, f_{r}\}$.
  Note that the space of Eisenstein series $E_{1,P}$ is one-dimensional
  and contained in $\Theta^{\sym}(P)$. Therefore, we can assume that all $f_{j}$ are cusp forms.
  Then $f_{j} = \sum_{\psi} a_{j,\psi} \Theta_{P}^{\psi} + g_{j}$ with $g_{j} \in \Theta^{\sym}(P)^{\perp}$
  as in the proof of Lemma \ref{lem:petersson-quot}.
  However, since $f_{j}^{\sigma} = f_{j}$ for all $j$ and all $\sigma \in \Gal(\C/\Q)$
  we have $a_{j,\psi}^{\sigma} = a_{j,\psi^{\sigma}}$ which shows that
  $g_{j}^{\sigma}=g_{j}$ and therefore, $g_{j}$ has also only rational Fourier coefficients.
  Thus $g_{j} \in M_{1,P}(\Q)$ and the set $\{g_{1}, \ldots, g_{r}\}$ forms a basis
  of $\Theta^{\sym}(P)^{\perp}$.
\end{proof}

Now Theorem \ref{thm:coeffs-intro} follows from Theorem \ref{thm:pos-coeffs} as follows.
Put $P=\fraka$ together with the quadratic form $Q(x)=\frac{\N x}{\N\fraka}$, as before.
Recall that $(\Theta_P(\tau,h))_{0}=\theta_{[\fraka][\fraka(h)]^{2}}(\tau)$ by Proposition \ref{lem:thetas}
and Section \ref{sec:GspinU}. This yields items (ii) and (iii) in the theorem.
The first property follows directly from the Sturm bound for $\Gamma_{0}(\abs{D})$,
we refer to Section \ref{sec:gross-zagi-situ}.
As for item (iv), this follows from the fact that $\pre{\Theta_{P}(\tau,h)} \in M_{1,P^{-}}$
with the property $\sum_{\beta \in P'/P} (\pre{\Theta_{P}(\tau,h)})_{\beta}) = (\pre{\Theta_{P}(\tau,h)})_{0}$
as $\Theta_{P}(\tau,h)$ has this property. Thus, there can only be non-vanishing coefficients
in the $0$-component that are represented by $NQ$ modulo $N$. This is reflected by the genus characters.

Requiring that $\pre{\Theta_{P}}^{\sym}(\tau,h) = \pre{E_P}(\tau) + \pre{g_{P}}^{\sym}(\tau,h)$ as in \eqref{SumG=E}
with $\pre{g_{P}}^{\sym}(\tau,h) \in H_{1,P^{-}}$ implies the formula for the constant coefficient
(note that $\xi(\pre{g_{P}})(\tau,h) \in S_{1,P}$. Thus, we have $\{ E_{P}, \pre{g_{P}}^{\sym}(\tau,h) \} = 0$).
Now, by Lemma \ref{lem:orthrat}, we can choose a basis of $M_{1,P}(\Q)$ which partitions into
bases of $\Theta^{\sym}(P)$ and $\Theta^{\sym}(P)^{\perp}$.
The shape of the principal part then follows from Lemma \ref{lem:preim_pp}
together with Proposition \ref{prop:pet-lift}.
The last item follows directly from Theorem \ref{thm:pos-coeffs} by considering only the $0$-component.
We simplified the expression a bit by taking the formulas for the coefficients $\kappa(n,\beta)$
from Theorem \ref{thm:EpreFourier} into account and collecting their contribution
in the term $\log\abs{\alpha([\fraka],n)}$.

\subsection{The principal part: Petersson norms}
\label{sec:princ-part:-peterss}
In this section we come back to the question on how to determine the principal part
of $\pre{\theta_{Q}}$ in the introduction. We will first develop a general formula and
then make this explicit for prime discriminants. The case of prime discriminants
has also been covered in \cite{DukeLiMock} using a different method.

We use the same assumptions as in the last section. Namely, we assume that
$(P,Q) \cong (\fraka,\frac{\N(x)}{\N(\fraka)})$ for a fractional ideal $\fraka \subset k_{D}$.
Furthermore, we assume that $K \cong \hat\calO_{D}^{*}$ which implies that $C_{P,K} \cong \Cl(D)$.
We want to develop a closed formula for the Petersson norms
of the orthogonal basis $\calB(P)$ by using a simple form of a seesaw dual pair.
We will use Lemma \ref{lem:orthrat} for this.

We have to determine the Petersson norm of $\Theta_{P}^{\psi}$ for non-trivial $\psi$.
We have that
\begin{align}
    (\Theta_{P}^{\psi}(\tau),\Theta_{P}^{\psi}(\tau))
  &= \int_{\SL_2(\Z) \bs \uhp} \langle \Theta_{P}^{\psi}(\tau), \overline{\Theta_{P}^{\psi}(\tau)} \rangle\, v\, d\mu(\tau)\\
  &= \sum_{h,h'} \psi(hh'^{-1}) \int_{\SL_2(\Z) \bs \uhp}^{\reg} \langle \Theta_{P}^{\sym}(\tau,h) ,  \overline{\Theta_{P}^{\sym}(\tau,h')} \rangle\, v\, d\mu(\tau).
     \label{petsum1}
\end{align}
Now, if we set $v_{P} = \sum_{\beta \in P'/P} \phi_{\beta} \otimes \phi_{\beta}$, then the last integral can be rewritten as
\begin{equation}
  \int_{\SL_2(\Z) \bs \uhp}^{\reg} \langle v_{P} , \overline{\Theta_{P}^{\sym}(\tau,h') \otimes \Theta_{P^{-}}^{\sym}(\tau,h)} \rangle\, d\mu(\tau) \label{ppliftconst1}
\end{equation}
We let $L=P \oplus P^{-}$ and denote by $\domain$ the symmetric domain for $V=L \otimes \Q$,
a rational quadratic space of type $(2,2)$. We identify the corresponding symmetric space with $\uhp \times \uhp$.
Moreover, let $z_{P} \in \domain$ be the CM point corresponding to $P \otimes \Q \subset V$.
We consider $O(P) \times O(P^{-})$ embedded into $O(P \oplus P^{-})$.
If we put
$$
  v_P^{\sym} = \sum_{\substack{\sigma \in \Aut(P'/P) \\ \mu \in \Aut(P'/P)}}
              \sum_{\beta \in P'/P} \phi_{\beta}^{\sigma} \otimes \phi_{\beta}^{\mu}
$$
then equation \eqref{ppliftconst1} becomes
\begin{equation}
  \int_{\SL_2(\Z) \bs \uhp}^{\reg} \langle v_{P}^{\sym} , \overline{\Theta_{P \oplus P^{-}}(\tau,z_{P},(h',h))} \rangle\, d\mu(\tau)
  = \Phi_{L}(z_{P},(h',h),v_{P}^{\sym}). \label{ppliftconst2}
\end{equation}
This is the regularized Borcherds lift for the lattice $P \oplus P^{-}$
of the constant $v_{P}^{\sym}$ evaluated at the CM point $z_{P}$.

It follows from Theorem 13.3 of \cite{boautgra} that there is a holomorphic modular form
$\Psi_{L}(z,(h',h),v_{P})$ for $\Gamma_{1} \times \Gamma_{2} \subset \SL_2(\Z) \times \SL_2(\Z)$
of parallel weight $1/2$, such that
\begin{equation}
  \Phi_{L}(z,(h',h),v_{P}^{\sym}) = -4 \log \abs{\Psi_{L}(z,(h',h),v_{P}^{\sym})} - \log(y_{1}y_{2}) + c
\end{equation}
for $z=(z_{1},z_{2})=(x_{1}+iy_{1},x_{2}+iy_{2}) \in \uhp \times \uhp$ and a constant $c \in \C$.
Here $\Gamma_{1}, \Gamma_{2}$ are congruence subgroups of $\SL_{2}(\Z)$.
This implies (see for instance Proposition 26 in Zagier's article in \cite{123}) that
\begin{equation}
  \Psi_{L}(z_{P},(h',h),v_{P}^{\sym}) \cdot (y_{1,P}y_{2,P})^{\frac{1}{4}} = \alpha_{P}(h,h') \Omega_{D}
\end{equation}
for a period $\Omega_{D}$ only depending on $D$, algebraic numbers $\alpha_{P}(h,h') \in \bar\Q$
and $z_{P}=(z_{1,P},z_{2,P}) = (x_{1,P}+i y_{1,P},x_{2,P}+iy_{2,P})$.
This yields:
\begin{proposition}
  Let $\psi \in \calC$ with $\psi \neq 1$. Then, there exist  algebraic numbers
  $\alpha_{P}(h) \in \bar{\Q}$ for $h \in C_{P,K}$, such that
\begin{equation}
  \label{eq:1}
  (\Theta_{P}^{\psi}(\tau),\Theta_{P}^{\psi}(\tau)) = -4 \sum_{h \in C_{P,K}} \psi(h) \log \abs{\alpha_{P}(h)}.
\end{equation}
\end{proposition}
\begin{proof}
  We have seen that
  \begin{equation*}
    (\Theta_{P}^{\psi}(\tau),\Theta_{P}^{\psi}(\tau))
    = -4 \sum_{h,h'} \psi(hh'^{-1}) \left( \log \abs{\alpha_{P}(h,h') \Omega_{D}} + c \right).
  \end{equation*}
  We substitute $k=h'h$ and put $\alpha_{P}(k) = \prod_{h \in C_{P,K}} \alpha_{P}(h,kh^{-1})$, which yields
  \begin{equation*}
    -4 \sum_{k \in C_{P,K}} \psi(k) \left( \log \abs{\alpha_{P}(k)} + h_{D}\log\abs{\Omega_{D}} + h_{D}c \right).
  \end{equation*}
  Since $\psi \neq 1$, this is the statement of the Proposition.
\end{proof}

Now we specialize to the case $D=-p$ for a prime $p \equiv 3 \bmod 4$.
We note that in this case $\Theta(\tau,h)$ is symmetric since the automorphism group of
$P'/P \cong \Z/p\Z$ is only $\{\pm 1\}$.

We consider the even unimodular lattice $L=M_2(\Z)$ with the quadratic form $Q(X)=-\det(X)$.
The discriminant kernel is given by $\SL_{2}(\Z)\times\SL_{2}(\Z)$.
The additive Borcherds lift of the constant function is equal to
$\Phi_{L}(z_{1},z_{2}) = -4 \log\abs{(y_{1}y_{2})^{1/4}\eta(z_{1})\eta(z_{2})} + c$ for a constant $c$.
We refer to Section 5.1 of the thesis of Hofmann \cite{ericdiss} for details.

We can embed $P \oplus P^{-}$ into $L$ as follows.
First of all, since $D$ is prime, there is only one genus of positive definite
binary quadratic forms of discriminant $D$ and therefore we can assume without loss of
generality that $P \cong \calO_{D}$ or, equivalently, that $P$ corresponds to
$\Z^{2}$ with binary quadratic form $[1,1,\frac{1-D}{4}]$.
We can define an embedding via
\begin{equation*}
  (1,0) \mapsto
    \begin{pmatrix}
    -1 & 1 \\ 1 & 0
  \end{pmatrix}
  \quad
  (0,1) \mapsto
  \begin{pmatrix}
    \frac{1-D}{4} & 0 \\ 0 & -1
  \end{pmatrix}.
\end{equation*}
Similarly for $P^{-}$, we get an embedding via
\begin{equation*}
  (1,0) \mapsto
    \begin{pmatrix}
    0 & 1 \\ -1 & 0
  \end{pmatrix}
  \quad
  (0,1) \mapsto
  \begin{pmatrix}
    \frac{1-D}{4} & -1 \\ 0 & 1
  \end{pmatrix}.
\end{equation*}
This maps $P \oplus P^{-}$ into $L$ as an orthogonal sum.
The point $z_{P}$ corresponds to $(\frac{1+\sqrt{D}}{4},\frac{1+\sqrt{D}}{4}) \in \uhp^{2}$
and $T(\adeles_{f})$ permutes all CM points of the form $(\alpha_{Q},\alpha_{Q})$
for $Q \in \calQ_{D}$.

\begin{lemma}
 Let $\psi=\chi^{2} \in \calC^{2}$ with $\psi \neq 1$. Then 
\begin{align*}
  w_{D}^{2} (\theta_{\chi}(\tau),\theta_{\chi}(\tau))
  =(\Theta_P^{\psi}(\tau),\Theta_P^{\psi}(\tau)) &= -4 h_{D} \sum_{Q \in \SL_{2}(\Z) \bs \calQ_{D}} \psi(Q) \log \abs{\sqrt{y_{Q}}\,\eta^{2}(\alpha_{Q})},
\end{align*}
where $\alpha_{Q}=x_{Q}+i y_{Q}$.
\end{lemma}
Note that this completely determines all the coefficients appearing in the principal part of $\pre{\theta_{Q}}$ since
$g_{Q}(\tau)=\sum_{\chi \in \calC} \chi(Q)\theta_{\psi}(\tau)$.
We should also remark that the values $\sqrt{y_{Q}}\,\eta^{2}(\alpha_{Q})$ are not algebraic but
they are, as we mentioned before, algebraic multiples of the period $\Omega_{D}$ which vanishes in the character sum.

\section{The Gross-Zagier situation: $n=1$} \label{sec:gross-zagi-situ}
We consider the rational quadratic space of signature $(2,1)$ given by
\[
V:=\left\{\lambda=\begin{pmatrix} \lambda_1 &\lambda_2\\\lambda_3& -\lambda_1\end{pmatrix}; \lambda_1,\lambda_2,\lambda_3 \in \Q\right\}
\]
and the quadratic form $Q(\lambda):=-\text{det}(\lambda)$.
The corresponding bilinear form is $(\lambda,\mu)=\text{tr}(\lambda \mu)$ for $\lambda, \mu \in V$.

In $V$, we consider the lattice
\[
L:=\left\{ \begin{pmatrix} b& c \\ -a&-b \end{pmatrix}; \quad a,b,c\in\Z \right\}.
\]
The dual lattice is given by
\[
L':=\left\{ \begin{pmatrix} b/2& c \\ -a&-b/2 \end{pmatrix}; \quad a,b,c\in\Z \right\}.
\]
We identify the discriminant group $L'/L=$ with $\Z/2\Z$,
together with the $\Q/\Z$ valued quadratic form ${x \mapsto x^2/4}$.
The level of $L$ is $4$.

For $\mu \in L'/L$, we let
\begin{equation}
  \label{eq:LD}
   L_{m,\mu} := \{ \lambda \in L + \mu\ \mid\ Q(\lambda)=m \}.
\end{equation}

The group $H = \GSpin(V)$ is isomorphic to $\GL_2$ as an algebraic group over $\Q$.
A matrix $\gamma \in \GL_2(\Q)$ acts via conjugation on $V$, i.e.
\[
\gamma.v = \gamma x \gamma^{-1}.
\]

The domain $\domain$ can be identified with $\uhp \cup \overline{\uhp}$ via
\[
z = x + iy \mapsto \R \Re \zmatrix \oplus \R \Im \zmatrix.
\]
Note that this gives indeed two different points for $x+iy$ and $x-iy$ since
the resulting two-dimensional spaces are the same but the given bases are
oriented differently.

For every prime $p$ consider the compact open subgroup $\GL_2(\Z_p) \subset H(\Q_{p})$ and
let $K = \prod_{p}\GL_2(\Z_p) \subset H(\adeles_{f})$. This preserves the lattice $L$ and
acts trivially on $L'/L$. By strong approximation, we have $H(\adeles_f) = H(\Q)K$ and therefore
we may identify the modular curve $\SL_2(\Z) \bs \uhp = Y(1)$ with the Shimura curve $X_K$
by Lemma \ref{lem:shimdecomp}.

The CM points, also called Heegner points or quadratic irrationalities in this situation,
are given as follows. Let $\lambda \in L'$ be a primitive vector
with $Q(\lambda) = \frac{b}{2}-ac < 0$. It determines a rational 2-dimensional
positive definite subspace $U=\lambda^{\perp} \subset V(\Q)$ and thus a pair of points
$\{Z(\lambda)^{\pm}\} \in \domain$.
We will simply write $Z(\lambda)$ for $Z(\lambda)^{+} \in \uhp$ and use the same notation for
the image of that point on the modular curve $Y(1)$.

\begin{lemma}
  For $\lambda = \smallabcmat$ the point $Z(\lambda)^{+}$ is given by
  \[
    Z(\lambda)^{+}= \frac{-b+\sqrt{D}}{2},
  \]
  where $D=b^2-4ac=4Q(\lambda)$.
\end{lemma}

Therefore, the point $Z(\lambda)^{+}$ is the same point as the one corresponding to the
integral binary quadratic form $[a,b,c]$.
If $h_{D}$ denotes the class number of $k_D=\Q(\sqrt{D})$,
there are exactly $h_{D}$ Heegner points on the modular curve $X_K$,
forming the cycle $Z(U)$ for any $\lambda$ of discriminant $D$ as above.
Note that we have $U \cong k_{D}$.

We have a rational splitting of $V=U+V_{-}$, where $V_{-}=\Q\lambda$.
As before, we consider the positive definite lattice
$P = L \cap U$ and the negative definite lattice $N = L \cap V_{-}$.
It is an explicit calculation to determine the lattices $P$ and $N$.
In the following, we will skip most of the proofs as they are simple calculations
and most of them can be found in Schofer's thesis \cite{schofer-thesis}.

\begin{lemma}
  \label{lem:PNlambda}
  Assume that $\lambda \in L + \mu$ is primitive and given by
  \[
    \lambda = \abcmat
  \]
  with $\mu \in L'/L = \Z/2\Z$.
  Then,
  \[
  N = L \cap V_{-} = \Z (1+\mu) \lambda.
  \]
  Moreover, the two-dimensional positive definite lattice $P = L \cap U$
  is given by
  \[
  P = \begin{cases}
    \frac{1}{2}l_{1}\Z \oplus l_{2} \Z, &\text{if } (a,c) \equiv 0 \pmod 2,\ \mu = 0,\\
    l_{1}\Z \oplus l_{2} \Z, &\text{otherwise}.
  \end{cases}
  \]
  Here,
  \begin{align*}
    l_1 &=
    \begin{pmatrix}
      (a,c) & bu \\
      bv    & -(a,c)
    \end{pmatrix},\\
    l_2 & =
    \begin{pmatrix}
      0 & \frac{a}{(a,c)} \\
      \frac{c}{(a,c)} & 0
    \end{pmatrix},
  \end{align*}
  where $u,v \in \Z$ with $av-uc=(a,c)$.
\end{lemma}

\begin{lemma}
  We have
  \[
  N' = \Z \frac{2}{(1+\mu)D} \lambda.
  \]
\end{lemma}

\begin{lemma}
  \label{lem:PNdiscOrder}
  We have
  \begin{align*}
    \abs{N'/N} &=
    \begin{cases}
      \abs{D}/2, & \text{if } \mu = 0,\\
      2\abs{D}, & \text{if } \mu =1.
    \end{cases}\\
    \abs{P'/P} &=
    \begin{cases}
      \abs{D}/4, & \text{if } (a,c) \equiv 0 \pmod 2 \text{ and } \mu=0,\\
      \abs{D}, & \text{otherwise}.
    \end{cases}
  \end{align*}
\end{lemma}

We will now show that we can choose $\lambda$ such that $P$
is isometric to the ring of integers $\calO_{D} \subset k_{D}$ together with
the quadratic form given by the norm $N(x)=x\bar{x}$ on $k_{D}$.
Then, we can choose $K_{P} = \hat\calO^\times_{D}$ so that $C_{P,K}$ is isomorphic to $Cl(D)$,
the class group of $k_{D}$ and for $h \in C_{P,K}$, the theta function $\Theta_P(\tau,h)$ can be
identified with the theta function of the ideal class $[h]^{2} \in Cl(D)$.

For $\mu \in \{0,1\}$, such that $D \equiv \mu \bmod 2$ we fix
\begin{equation}
  \label{eq:lambdaO}
  \lambda_D =
  \begin{pmatrix}
    \frac{\mu}{2} & -1 \\ \frac{\mu-D}{4}  & -\frac{\mu}{2}
  \end{pmatrix}.
\end{equation}
Then $Q(\lambda)=D/4$ and $\lambda_D$ corresponds to the binary quadratic form $[1,\mu,(\mu-D)/4]$,
which lies in the trivial class. We use the same notation as before, but now everything depends on our
particular choice $\lambda_{D}$ (i.e. $U=\lambda_D^{\perp}$ from now on).

\begin{lemma}
  \label{lem:PO}
  The following map is an isometry of quadratic lattices:
  \begin{equation*}
    f:(\calO,\N) \rightarrow (P,Q),\quad x + y\frac{\mu+\sqrt{D}}{2} \mapsto
    \begin{pmatrix}
      x & -y \\ -\mu x - y \frac{\mu - D}{4} & -x
    \end{pmatrix}.
  \end{equation*}
  Moreover, both are equivalent to the integral quadratic form
  $[1,\mu,\frac{\mu-D}{4}]$.
\end{lemma}
\begin{proof}
  This is a special case of \cite[Lemma 7.1]{bryfaltings}.
\end{proof}

\emph{From now on, we assume that $D \equiv 1 \bmod 4$ is an odd, negative fundamental discriminant.}
We can identify the vector-valued theta functions with their scalar-valued analogs (see also section \ref{sec:proof-thm2}).
Let us summarize the previous lemmas in this case.
We have
\begin{align}
  \label{eq:Nodd}
  N &= \Z 2\lambda_D, \quad N' = \Z \frac{1}{D}\lambda_D,\\
  P &= \Z \begin{pmatrix} 1 & 0 \\ 1 & -1 \end{pmatrix} \oplus \Z \begin{pmatrix} 0 & 1  \\ \frac{1-D}{4} & 0 \end{pmatrix}.
\end{align}
And by Lemma \ref{lem:PNdiscOrder}, $\abs{N'/N}=2\abs{D}$ and $\abs{P'/P}=\abs{D}$.
Clearly, in this case $P'/P$ is cyclic since $D$ is square-free.

Now we would like to determine the point in $Y(1)$ corresponding to
the double coset $[Z(\lambda_D),h]=H(\Q)(Z(\lambda_{D}),h)K$ for $h \in T(\adeles_f)$.
For this, we proceed as follows. We will compute $\gamma \in H(\Q)^{+}$ and
$k \in K$ such that $h=\gamma k$ and we get $[Z(\lambda_D),h] = [\gamma^{-1}Z(\lambda_D),1]$.
This, in turn, corresponds to $\SL_2(\Z)\gamma^{-1}Z(\lambda_D)$.

\begin{lemma}
  Let $\lambda = \smallabcmat$ with $Q(\lambda)=D/4$.
  Moreover, let
  \begin{equation}
    \label{eq:gl}
     \gamma_{\lambda} :=
        \begin{pmatrix}
          1 &  0 \\ \frac{b-1}{2} & a
        \end{pmatrix}.
  \end{equation}
  Then $\gamma_{\lambda}.\lambda = \lambda_{D}$.
\end{lemma}

\begin{lemma}
  The group $K_p= \GL_2(\Z_p)$ acts transitively on $L_{D/4,1}$.
  
  More explicitly, let $\lambda = \smallabcmat \in L_{D/4,1}^{0}$.
  \begin{enumerate}
   \item If $(p,a)=1$, we have $\gamma_{\lambda} \in \GL_2(\Z_p)$ and we put $k_{\lambda,p} = \gamma_{\lambda}^{-1}$.
   \item If $p \mid (a,c)$, we have $(p,b)=1$. Let
        \begin{equation}
         \label{eq:kl1}
         k_{\lambda,p} := \frac{1}{a+b+c}
         \begin{pmatrix}
           \frac{1+2a+b}{2} & -1 \\ \frac{2c+b-1}{2} & 1
         \end{pmatrix}.
        \end{equation}
   \item If $p \mid a$ but $(p,c)=1$ put
     \begin{equation}
         \label{eq:kl2}
         k_{\lambda,p} :=
         \begin{pmatrix}
           -\frac{b+1}{2c} & \frac{1}{c} \\ -1 & 0
         \end{pmatrix}.
     \end{equation}
  \end{enumerate}
  Then, in any case, $k_{\lambda,p} \in K_p$ and $k_{\lambda,p}.\lambda_{D} = \lambda$.
\end{lemma}

We put $k_{\lambda}=(k_{\lambda,p})_{p} \in K$.
The two preceding lemmas imply that $\gamma_{\lambda}k_{\lambda} \in T(\adeles_f)$.
We will now determine the ideal class that corresponds to $\gamma_{\lambda}(k_{\lambda,p})_{p}$.

\begin{lemma}
  The isomorphism  $T \cong k^{\times}$ can be explicitly realized by
  $\Xi: T \rightarrow k^{\times}$ via
  \begin{align*}
    \begin{pmatrix}
      1 & 0 \\ 0 & 1
    \end{pmatrix}
      &\mapsto 1, \\
      \begin{pmatrix}
        1 & -2 \\ \frac{1-D}{2} & -1
      \end{pmatrix}
      & \mapsto \sqrt{D}.
  \end{align*}
\end{lemma}

\begin{proposition}
  If $(p,a)=1$ then $\Xi(\gamma_{\lambda}k_{\lambda})_{p}=1$.
  If $p \mid (a,c)$ then
  \[
    \Xi(\gamma_{\lambda}k_{\lambda})_{p} =  \frac{1}{a+b+c} \left( a + \frac{b+\sqrt{D}}{2} \right) \in k_{D} \subset k_D \otimes \Q_{p}.
  \]
  Finally, if $p \mid a$ but $(p,c)=1$ then
  \[
    \Xi(\gamma_{\lambda}k_{\lambda})_{p} =  \frac{-1}{c} \left( \frac{b+\sqrt{D}}{2} \right) \in k_{D} \subset k_D \otimes \Q_{p}.
  \]

  In particular, in all cases we get that $\nu_{\frakp}(\Xi(\gamma_{\lambda}k_{\lambda}))=\nu_{\frakp}(\fraka_{\lambda})$,
  where $\fraka_{\lambda} = (a,\frac{b+\sqrt{D}}{2}) \subset \calO_D$ is the ideal corresponding to the binary quadratic form
  $\psi(\lambda)$.
\end{proposition}

\begin{corollary}
  The map sending $L_{D/4,1}^{0} \rightarrow C_{P,K} \cong \Cl(k_D)$, mapping $\lambda$ to
  the class of $\gamma_{\lambda}k_{\lambda}$ is surjective. Moreover, given $[h] \in C_{P,K}$
  with $[\fraka(h)] = [\fraka_{\lambda}]$, then $[h]=[\gamma_{\lambda}k_{\lambda}]$ and thus
  we have that $[Z(\lambda),h] \in X_K$ corresponds to
  $\gamma_{\lambda}^{-1}Z(\lambda_D)=Z(\lambda) \in Y(1)$.
\end{corollary}

In order to determine $F_{P\oplus N}$ for an input function $F \in M^{!}_{\frac{1}{2},L}$,
we have to do some more explicit lattice computations.

\begin{lemma}
  \label{lem:Doddcyc}
  Assume that $D$ is an odd fundamental discriminant.
  Then $L'/(P \oplus N)$ is cyclic of order $2\abs{D}$.
  It is generated by the image of
  \[
  x=
  \begin{pmatrix}
    \frac{1}{2} & 0 \\ 0 & -\frac{1}{2}
  \end{pmatrix}.
  \]
  Moreover, we have
  \[
    \begin{pmatrix}
      \frac{1}{2} & 0 \\ 0 & -\frac{1}{2}
    \end{pmatrix}
    =
    \begin{pmatrix}
      \frac{D-1}{2D} & \frac{-1}{D} \\ \frac{1-D}{4D} & \frac{1-D}{2D}
    \end{pmatrix}
    + \frac{1}{D}\lambda.
  \]
  We identify $P'/P \oplus N'/N$ with the finite quadratic module
  $\Z/\abs{D}\Z \oplus \Z/2\abs{D}\Z$ with quadratic form $\frac{x^{2}}{\abs{D}}-\frac{y^{2}}{4\abs{D}}$.
  Then, if $D$ is an odd fundamental discriminant,
  the element $x$ corresponds to one of the two elements $(\pm t,1)$, where
  $2t \equiv 1 \bmod \abs{D}$.
\end{lemma}
We will see that the sign above does not concern our computations.

Even more general, a weak Maa\ss{} form $F \in H_{1,P}$
corresponds to a scalar-valued weak Maa\ss{} form via
\[
  F = \sum_{\beta \in P'/P} F_{\beta}(\tau) \phi_{\beta} \mapsto \sum_{\beta \in P'/P} F_{\beta}(\abs{D}\tau).
\]
In view of Lemma \ref{lem:PO},  $\Theta_P(\tau,1)$ corresponds to
$\theta_{[\calO]}(\tau)$, the theta function of the trivial class in the class group of $k_{D}$.
Moreover, from our calculation of the action of $h \in C_{P,K}$ in
\eqref{eq:classha}, it follows that if $[\fraka(h)]$ denotes the ideal class
corresponding to $h$, then $\Theta_P(\tau,h)$ corresponds to
$\theta_{[\fraka(h)]^{2}}(\tau)$. On the other hand, the CM point (in $Y(1)=\uhp/\SL_2(\Z)$)
corresponding to $h$ is exactly the one that corresponds to the ideal class $[\fraka(h)]$.
In other words, if we take an ideal $\fraka \subset \calO$ corresponding
to an integral binary quadratic form $[a,b,c]$ of discriminant
$D < 0$, and evaluate the theta lift $\Phi(\tau,F)$ at $\tau=\alpha_\fraka=\frac{-b+\sqrt{D}}{2a}$,
Theorem \ref{thm:value-phiz} gives us a formula for this value in terms
of the coefficients of $F$ and of a preimage of the
theta function $\theta_{[\fraka]^{2}}(\tau)$ under $\xi_{1}$. Here, $[\fraka]^{2}$ denotes the square class
of $\fraka$. We will freely work with these correspondences from now on.

Now we can proceed determining $F_{P\oplus N}$.
The space $M^{!}_{1/2,L}$ is isomorphic to $M_{1/2}(4)^{!,+}$ of scalar-valued
weakly holomorphic modular forms on $\Gamma_0(4)$ in the Kohnen plus-space,
which consists of forms $f$ which have only non-vanishing Fourier coefficients
for indices $n$ with $n \equiv 0,1 \pmod 4$.

For $f=\sum_{n \in  \Z}c(n)q^n$, the isomorphism is given by
\begin{equation}
  \label{eq:5}
  f \mapsto F = \sum_{\substack{n \in \Z \\ n \equiv 0 \bmod 4}} c(n)q^{n/4}   \phi_{0}
  + \sum_{\substack{n \in \Z \\ n \equiv 1 \bmod 4}} c(n)q^{n/4}   \phi_{1}
  = F_{0} \phi_{0} + F_{1} \phi_{1}.
\end{equation}
Moreover, for $F=F_{0} \phi_{0} + F_{1} \phi_{1}$, by the calculations made above,
\begin{equation}
  \label{eq:6}
  F_{P \oplus N} = \sum_{y \in \Z/2\abs{D}\Z} F_{y} \phi_{y(\pm t,1)}.
\end{equation}

For a positive integer $-d \equiv 0,3 \bmod 4$,
there is a unique weakly holomorphic modular form
$f_{d}=q^{-d}+O(q) \in M_{1/2}^{!}(4)$. Its image under the theta lift is given by
(compare Theorem 6.1 and Section 8.1 in \cite{bronolderiv})
\[
  \Phi(\tau,F_{d}) = -4 \log \abs{ \prod_{\frakb \in \Cl(-d)} \left( j(\tau) - j(\alpha_{\frakb})  \right) }^{2/w_{d}}.
\]
Therefore, comparing with Theorem \ref{thm:value-phiz}
for the value at a CM point $\alpha_{\fraka}$ of discriminant $D$
we obtain the equality
\begin{equation}
  \label{eq:7}
  \Phi(\alpha_{\fraka},F_{d}) = \CT(\langle (F_{d})_{P\oplus N}(\tau), \Theta_{N^{-}}(\tau) \otimes \prep{\Theta_{P}}(\tau,h) \rangle)
  = -4 \log \abs{ \prod_{\frakb \in \Cl(-d)} \left( j(\alpha_{\fraka}) - j(\alpha_{\frakb}) \right)}^{2/w_{d}},
\end{equation}
where $h \in C_{P,K}$ corresponds to $[\fraka]$.

We write the Fourier expansion of $\prep{\Theta_{P}}(\tau,h)$ as
\[
\prep{\Theta_{P}}(\tau,h) = \sum_{\beta \in P'/P} \prep{\Theta_{P,\beta}}(\tau,h)  \phi_{\beta}
  = \sum_{\beta \in P'/P} \sum_{m \in \Q} c^{+}_{P,h}(m,\beta) e(m\tau) \phi_{\beta}.
\]
Note that $\prep{\Theta_{P,\beta}}(\tau,h) = \prep{\Theta_{P,-\beta}}(\tau,h)$, which is
the reason why we do not need to worry about the sign of $t$.

We insert this and the Fourier expansion of $F_d$ into \eqref{eq:7}, so that the constant term in \eqref{eq:7} becomes
\begin{equation}
  \label{eq:cmval1}
  \sum_{n \equiv 0,1 \bmod 4} f_{d}(n) \sum_{\substack{\mu \in \Z/2\abs{D}\Z \\ \mu \equiv n \bmod 2}}\ \delta(\mu)
      \sum_{\substack{m \in \Z \\ m \equiv \mu \bmod {2\abs{D}}}} c_{P,h}^{+}\left(\frac{nD-m^{2}}{4\abs{D}},t \mu\right),
\end{equation}
where
\begin{equation}
  \label{eq:deltamu}
  \delta(\mu) =
    \begin{cases}
      2, & \text{if } \mu \equiv 0 \bmod \abs{D},\\
      1, & \text{if } \mu \not\equiv 0 \bmod \abs{D}
    \end{cases}
\end{equation}
comes from the coefficient of $\Theta_{N}$.

We identify the function $\pre{\Theta_{P}}(\tau,h)$ with the scalar-valued
weak Maa\ss{} form $\pre{\theta_{[\fraka(h)]^{2}}}(\tau)$ in the space $\calH_1(\Gamma_0(\abs{D}),\chi_D)$.
It is a preimage under $\xi_{1}$ of the scalar-valued theta function
$\theta_{[\fraka(h)]^{2}}(\tau) \in \Theta(D) \subset  M_1^{+}(\Gamma_0(\abs{D}),\chi_D)$.

We write the Fourier expansion of the holomorphic part of $\pre{\theta_{[\fraka(h)]}}(\tau)$ as
\[
 \prep{\theta_{[\fraka(h)]^{2}}}(\tau) = \sum_{n \in \Z} c_{[\fraka(h)]^{2}}^{+}(n) e(n\tau).
\]
Inserting this into \eqref{eq:cmval1} gives us a formula to which only
the scalar-valued input function $f_d$ and the holomorphic part of the
scalar-valued Maa\ss{} form $\pre{\theta}_{[h]}(\tau)$ contribute:
\begin{equation}
  \label{eq:cmval2}
   \sum_{n \equiv 0,1 \bmod 4} f_{d}(n)
      \sum_{m \in \Z} \delta(m)\, c_{[\fraka(h)]^{2}}^{+}\left(\frac{ nD - m^{2}}{4}\right).
\end{equation}
Note that up to this point, everything we did holds for \emph{any} weak Maa\ss{} form
$\pre{\theta_{[\fraka(h)]^{2}}}$ with $\xi_1\pre{\theta_{[\fraka(h)]^{2}}}=\theta_{[\fraka(h)]^{2}}$.

Nevertheless, in accordance with our efforts to determine
a ``nice'' preimage, we will now specify $\pre{\theta_{[\fraka(h)]^{2}}}$
more explicitly.

The Sturm bound for $M_1(\Gamma_0(\abs{D}),\chi_D)$ is given by
$\frac{\abs{D}}{12} \prod_{p \mid D}\left( 1 + \frac{1}{p} \right)$.
If we suppose that $n>0$, then
\[
  \frac{n\abs{D}+m^{2}}{4} \geq \frac{\abs{D}}{4},
\]
which implies that if we assume that
\begin{equation}
  \label{eq:Dass}
    \prod_{p \mid \abs{D}}\left( 1 + \frac{1}{p} \right) \leq 3,
\end{equation}
we may \emph{choose} the principal part of
$\pre{\theta_{[\fraka(h)]}}$ such that $c_{[\fraka(h)]}^{+}(\frac{nD-m^{2}}{4})=0$
for all $n>0$ and $m \in \Z$.
This simplifies formula \eqref{eq:cmval2} and
we obtain Theorem \ref{thm:GZ-intro} of the introduction.

\bibliographystyle{alpha}
\bibliography{cmhfone}

\end{document}